\long\global\def\C#1\F{{}}

\documentclass[authoryear,12pt]{article}

\usepackage{natbib}

\usepackage[usenames,dvipsnames]{color}
\usepackage{amsmath, amsthm, amssymb, enumerate}

\usepackage{tikz}
\usetikzlibrary{decorations,decorations.pathmorphing,decorations.pathreplacing,calligraphy,decorations.markings}
\usetikzlibrary{arrows,backgrounds,calc}
\usetikzlibrary{shapes.symbols,shapes.geometric}
\tikzset{->-/.style={decoration={
  markings,
  mark=at position #1 with {\arrow{triangle 45}}}, postaction={decorate}}}
  
\usepackage{enumerate}
   
\theoremstyle{plain}
\newtheorem{proposition}{Proposition}[section]
\newtheorem{corollary}[proposition]{Corollary}
\newtheorem{theorem}[proposition]{Theorem}
\newtheorem{lemma}[proposition]{Lemma}

\theoremstyle{definition}

\newtheorem{definition}[proposition]{Definition}
\newtheorem{example}[proposition]{Example}
\newtheorem{remark}[proposition]{Remark}

\def\st{{\;\vrule height9pt width0.9pt depth2.5pt\;}}

\def\sqr#1#2{{\vcenter{\vbox{\hrule height.#2pt
         \hbox{\vrule width.#2pt height#1pt \kern#1pt
             \vrule width.#2pt}
          \hrule height.#2pt}}}}

\def\roster{\begin{enumerate}}
\def\endroster{\end{enumerate}}

\newcommand{\specialP}[2]{\PPP_{\mathrm#1}^{#2}}
\newcommand{\Pmc}[1]{\specialP{MC}{#1}} 

\newcommand{\specialD}[2]{D_{\mathrm#1}^{#2}}
\newcommand{\Dlo}[1]{\specialD{LO}{#1}} 
\newcommand{\Dwo}[1]{\specialD{WO}{#1}} 
\newcommand{\Dso}[1]{\specialD{SO}{#1}} 
\newcommand{\Dio}[1]{\specialD{IO}{#1}} 

\newcommand{\Flo}[1]{\FFF(\Dlo#1)} %
\newcommand{\Fwo}[1]{\FFF(\Dwo#1)} %
\newcommand{\Fio}[1]{\FFF(\Dio#1)} %
\newcommand{\Fso}[1]{\FFF(\Dso#1)} %

\newcommand{\conv}{\mathop{\mathrm{conv}}}
\newcommand{\cor}{\mathop{\mathrm{cor}}}

\newcommand{\R}{\mathbb R{}}

\def\st{{\;\vrule height9pt width1pt depth1.5pt\;}}

\def\CCC{{\cal C}}

\def\FFF{{\cal F}}

\def\LLL{{\cal L}}

\def\OOO{{\cal O}}
\def\PPP{{\cal P}}
\def\SSS{{\cal S}}

\def\conv{\text{conv}}

\def\es{\varnothing{}}

\def\z{\nobreak\hspace{.3em plus .08333em}}

\mathchardef\ordinarycolon\mathcode`\:
\mathcode`\:=\string"8000
\begingroup \catcode`\:=\active
  \gdef:{\mathrel{\mathop\ordinarycolon}}
\endgroup

\title{Adjacencies on random ordering polytopes and flow polytopes}

\author{Jean-Paul Doignon\\
D\'epartement de Math\'ematique, c.p.~216,\\ Universit\'e Libre de Bruxelles, Bruxelles, Belgium\\
\texttt{Jean-Paul.Doignon@ulb.be}
\\
\and 
Kota Saito\footnote{Saito acknowledges the financial support of the NSF through grants SES-1919263.}\\
HSS, California Institute of Technology,\\ Pasadena, CA, USA\\
\texttt{{saito@caltech.edu}}
}

\begin{document}

\maketitle

\begin{abstract}
The Multiple Choice Polytope (MCP) is the prediction range of a random utility model due to Block and Marschak (1960).  Fishburn (1998) offers a nice survey of the findings on random utility models at the time.  A complete characterization of the MCP is a remarkable achievement of Falmagne (1978).  Apart for a recognition of the facets by Suck (2002), the geometric structure of the MCP was apparently not much investigated.  Recently, Chang, Narita and Saito (2022) refer to the adjacency of vertices while Turansick (2022) 
uses a condition which we show to be equivalent to the non-adjacency of two vertices.  We characterize the adjacency of vertices and the adjacency of facets.  To derive a more enlightening proof of Falmagne Theorem and of Suck result, Fiorini (2004) assimilates the MCP with the flow polytope of some acyclic network.  Our results on adjacencies also hold for the flow polytope of any acyclic network.  
In particular, they apply not only to the MCP, but also to three polytopes which Davis-Stober, Doignon, Fiorini, Glineur and Regenwetter (2018) introduced as extended formulations of the weak order polytope, interval order polytope and semiorder polytope (the prediction ranges of other models, see for instance Fishburn and Falmagne, 1989, and Marley and Regenwetter, 2017).
\end{abstract}

\section{Introduction}

\cite{Block_Marschak1960} introduce ``random utility models'', showing in many cases their equivalence  with ``random ordering models''. In particular, the Multiple Choice Model (MCM) predicts stochastic choices from latent probability distributions over strict rankings; all sets of alternatives are choice sets, and the subject selects one alternative in the choice set\footnote{Other random utility models restrict choice sets, for instance to two-element sets.
In economics, the term  ``random utility model'' refers to models based on  probability distributions over strict rankings, that is irreflexive linear orderings.  
In psychology, relations of another type often replace rankings (see for instance the references in \citealp{Davis-Stober_Doignon_Fiorini_Glineur_Regenwetter2018}).} (for a precise definition, see Section~\ref{SEC_MCP}).

A complete characterization of the MCM is a remarkable result due to \cite{Falmagne1978}: the predictions of the MCM form the Multiple Choice Polytope (MCP), for which Falmagne obtains an affine description---that is, a system of affine inequalities whose solution set is the MCP.  

In economics, since  \cite{Marschak1960} and \cite{Block_Marschak1960}, the MCM has been used in many different contexts.  In discrete choice analysis, economists often use the MCM  to describe unknown data generating process of stochastic choice, for instance over transportation methods, schools, and products (although in practice, they frequently make use of   parametric models such as the mixed logit model, \citealp{McFadden2001}).  The interest for the MCM is exemplified by \cite{McFadden_Richter1970,McFadden_Richter1990}\footnote{McFadden and Richter establish another characterization of the model (a more involved one than Falmagne's one).},
\cite{Barbera_Pattanaik1986}\footnote{Barbera and Pattanaik obtain a proof similar to Falmagne's one.} and \cite{Monderer1992}\footnote{Monderer derives another proof from a result of \citealp{Weber1988} in game theory, namely a characterization of random order values.}.

In psychology, several papers refer to Falmagne Theorem, for instance \cite{Regenwetter_Marley_Grofman2002}, \cite{Suck2002b}, \cite{Fiorini2004}, \cite{Suck2016}. Recently, \cite{Kellen_Winiger_Dunn_Singmann2021} use the MCM in signal detection theory.

In both psychology and economics, and also in
operations research, another
setup in which the only choice sets are binary
is the object of many publications: see \cite{Fishburn1992} for a classical survey, and \cite{Marti_Reinelt2011} for a more recent overview.
For example, \cite{Fishburn_Falmagne1989} provide necessary conditions for binary choice probabilities to be induced by a probability distribution on rankings.  They also show that no finite set of simple necessary conditions is sufficient for inducement when the alternative set is finite but can be arbitrarily large. Today, finding a manageable characterization of the binary choice polytope appears to be out of reach in view of a related NP-hard problem (see for instance \citealp{Charon_Hudry2010}, Problem~5 and Theorem~7).

For the MCP, \cite{Fiorini2004} provides an alternative proof of Falmagne Theorem, which  is enlightening:  he starts with a change of space coordinates or, in another interpretation, he works on the image of MCP by a well-chosen affine transformation.  
Next he shows that in the new viewpoint the vertices of MCP are (the characteristic vectors of) all paths from the source to the sink in a special network.  Hence, the MCP is the flow polytope of the network.  A characterization of the MCP by a system of affine inequalities then follows from the fundamental theorem on network flows (\citealp{Gallai1958} and \citealp{Ford_Fulkerson1962}).  In Economics, \cite{Chambers_Masatlioglu_Turansick2021}
apply Fiorini's technique to study a ``correlated random utility model''.

However, not much is known about the geometric structure of the MCP other than its facets \citep{Suck2002}.  
We characterize the adjacency of vertices and the adjacency of facets. As a matter of fact, our characterizations hold
for the flow polytope of any acyclic network (the MCP being a particular case). 
So they are also valid for the three flow polytopes built in \cite{Davis-Stober_Doignon_Fiorini_Glineur_Regenwetter2018} to get extended formulations of the weak order polytope, interval order polytope and semiorder polytope\footnote{We refer the reader to the last paper (and its references) for the  terminology.
 Note that the mastery of the adjacencies on the  four
 extended formulations should be useful in the design of optimization algorithms, particularly for the statistical tests evoked in 
\cite{Davis-Stober_Doignon_Fiorini_Glineur_Regenwetter2018}.} (see Figure~\ref{scheme_of_polytopes}).  In Economics, \cite{Turansick2022}, in his Theorem~2 on the identifiability in the MCM (see \citealp{Fishburn1998}, for previous results),
introduces a condition on two vertices of the MCP which we show to be equivalent to their non-adjacency (see Subsection~\ref{subs_identifiability}). 
To check whether the mixed logit model can approximate the MCM, \cite{Chang_Narita_Saito2022} use the fact that 
 a convex combination between two adjacent vertices of the MCP is a prediction of the MCM that is uniquely represented.  Thus a characterization of vertex adjacency can be useful.

Fishburn published papers on the linear ordering polytope, notably \cite{Fishburn_Falmagne1989} and \cite{Fishburn1992}, and also on the weak order polytope, \cite{Fiorini_Fishburn2004}.
He has also introduced the concept of an interval order \citep{Fishburn1970a} as an extension of the one of a semiorder \citep{Luce1956}. We dedicate our contribution to the memory of Peter Fishburn, whose influence on the fields addressed in this paper remains so strong.

\begin{figure}[ht!]
\begin{center}
\begin{tikzpicture}[xscale=0.7,yscale=0.8,every text node part/.style={align=center}]
  
\draw [thick,decorate,decoration={brace,amplitude=15pt,raise=4ex}] (-2.1,3) --  (13.2,3) node[midway,yshift=16mm]{flow polytopes};

\draw [thick,decorate,decoration={brace,amplitude=9pt,raise=4ex}] (12.7,3.7) --  (12.7,-1) node[midway,xshift=17mm,rotate=-90]{extended\\ formulations};

\node[rectangle] at (0,3) {$\Pmc\CCC=\Flo\CCC$};
\node[rectangle] at (0,0) {linear \\ order \\ polytope};

\node[rectangle] at (4,3) {$\Fwo\CCC$};
\node[rectangle] at (4,0) {weak \\ order \\ polytope};

\node[rectangle] at (8,3) {$\Fio\CCC$};
\node[rectangle] at (8,0) {interval \\ order \\ polytope};

\node[rectangle] at (12,3) {$\Fso\CCC$};
\node[rectangle] at (12,0) {semi-\\ order \\ polytope};

\draw[thick,->] (0,2.3) to (0,1.3);
\draw[thick,->] (4,2.3) to (4,1.3);
\draw[thick,->] (8,2.3) to (8,1.3);
\draw[thick,->] (12,2.3) to (12,1.3);

\end{tikzpicture}
\end{center}
\caption{A scheme of the various polytopes mentioned in the paper.    Here $\Pmc\CCC$ designates the Multiple Choice Polytope MCP on the alternative set $\CCC$ (Section~\ref{SEC_MCP}), and $\FFF(D)$ designates the flow polytope of the network $D$ (see Sections~\ref {SE_MCP_flows} and \ref{SE_other} for the four specific networks). 
\label{scheme_of_polytopes}}
\end{figure}
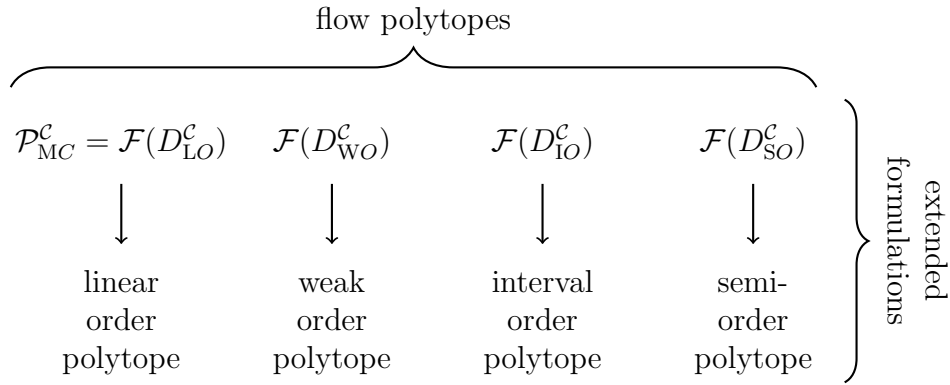

\section{Basic Definitions and Results}
\label{SEC_Basic}
\subsection{Polytopes}\label{sub_Polytopes}
A \textsl{polytope} $\PPP$ in $\R^d$ is the convex hull of some finite subset of $\R^d$, say $\PPP=\conv(V)$ with $V \subset \R^d$, $V$ finite.  
A \textsl{face} $F$ of the polytope $\PPP$ is any subset $F$ of $\PPP$ equal to $\PPP$, or for which there exists an (affine) hyperplane $H$ which satisfies $\PPP \cap H = F$ and is \textsl{valid} for $\PPP$, that is, $\PPP \subseteq H^+$ with $H^+$ a closed side of $H$. 
If $H^+=\{p \in \R^d \st \alpha(p)\ge(r)\}$ for a linear form $\alpha$ on $\R^d$ and a real number $r$, 
the inequality $\alpha(x)\ge r$  \textsl{defines} the face $F$.  
A \textsl{vertex} of $\PPP$ is a point $p$ such that $\{p\}$ is a face of $\PPP$.  An \textsl{edge} is a segment which forms a face.  A \textsl{facet} of $\PPP$ is a proper\footnote{Recall that $A$ is a \textsl{proper} subset of $B$ when $A \subset B$ (strict inclusion).}, maximal face of $\PPP$.

For our polytope $\PPP=\conv(V)$, all vertices belong to $V$ (but points in $V$ are not necessarily vertices).  Even more, the vertices form the single, inclusion-minimal subset $V$ such that $\PPP=\conv(V)$.  Any face is the convex hull of the vertices it contains. A \textsl{simplex} is a polytope whose vertices are affinely independent points.

Each polytope $\PPP$ in $\R^d$ is the set of solutions of a (finite) system $\SSS$ of affine equations and affine inequalities on $\R^d$.  Under the restriction that the solution set is bounded, the converse does hold.  The system $\SSS$ then forms an \textsl{affine description} of the polytope.  Suppose now that $\SSS$ is an affine description with a minimum number of (in)equalities.  If any inequality in $\SSS$ is satisfied with equality on the whole polytope $\PPP$, we replace the inequality sign with an equality sign.  Then the number of equalities in $\SSS$ equals the codimension of $\PPP$ (that is, $d-\dim(\PPP)$, where $\dim$ always means the affine dimension).  Moreover, there is in $\SSS$ one inequality per facet of $\PPP$. 
When $\dim \PPP < d$, the affine inequality for a given facet can be chosen among infinitely many ones.

For more details (especially proofs) on polytopes, see for instance \cite{Korte_Vygen2008}, \cite{Schrijver2003}, \cite{Ziegler1998}.

\subsection{Directed graphs}
A \textsl{directed graph} $G$ is a pair $(N,A)$, where $N$ is a finite set of nodes\footnote{We reserve the word ``vertex'' for polytopes.  In only a few other occasions when speaking of directed graphs, we depart  from the exposition of \cite{Bang-Jensen_Gutin2001}.} and $A$ is a set of arcs, each arc being a pair of distinct nodes (the definition excludes loops as well as parallel arcs).  For any arc $a=(u,v)$, we call $u$ the \textsl{tail} and $v$ the \textsl{head} of the arc $a$.

Let $G=(N,A)$ be a directed graph.  A \textsl{walk} in $G$ is a finite sequence $(u_1,v_1)$, $(u_2,v_2)$, \dots, $(u_k,v_k)$ of arcs with $k \ge 1$,  $v_{i-1}=u_i$ for  $i=2$, $3$, \dots, $k$.  The latter walk \textsl{starts} at its \textsl{initial node} $u_1$ and \textsl{ends} at its \textsl{terminal node} $v_k$, it is \textsl{from} $u_1$ to $v_k$.  
It \textsl{passes} through its \textsl{internal nodes}
$u_2$, $u_3$, \dots, $u_k$.  The walk is a \textsl{path} when its nodes are two by two distinct.  
A \textsl{cycle} in $G$ has a definition similar to the one of a path, except that $u_1=v_k$ is required. 

A directed graph is \textsl{acyclic} if it does not possess any cycle.  In an acyclic graph $(N,A)$, any walk is a path because any acyclic graph has a so-called \textsl{topological sort}, 
that is a linear ordering $L$ of its nodes such that for any arc $(u,v)$ there holds $u >_L v$.  Although paths are by definition sequences of arcs, we often treat them as sets of arcs (for instance when we say that a path includes another one).  In an acyclic graph, the set of arcs in a path determines in a unique way the path (as a sequence of these arcs).

Any set $B$ of arcs from $A$ (for example, $B$ is the set of arcs in a path) has its \textsl{characteristic vector} $\chi^B$ in $\mathbb{R}^{A}$: for any arc $a$ in $A$, we set $\chi^B(a) = 1$ if $a \in B$ and $\chi^B(a) = 0$ if $a \in A \setminus B$. 
For a point $x$ in $\mathbb{R}^{A}$ and $B \subseteq A$, define the number
\begin{equation}
x(B) := \sum_{a \in B} x(a).
\end{equation}
For each node $v$, we denote the sets of arcs with either head or tail $v$ by $\delta^-(v)$ and $\delta^+(v)$, respectively:
\begin{eqnarray*}
\delta^-(v) &:= &\{a \in A \st \exists u \in N : a = (u,v)\},\\
\delta^+(v) &:= &\{a \in A \st \exists w \in N : a = (v,w)\},
\end{eqnarray*}
and define the \textsl{in-degree} and \textsl{out-degree} of $v$ by
\begin{eqnarray*}
d^-(v) &:= &|\delta^-(v)|,\\
d^+(v) &:= &|\delta^+(v)|.
\end{eqnarray*}

\subsection{Network Flows}
A \textsl{network} $D = (N,A,s,t)$ is\footnote{Here we follow \cite{Korte_Vygen2008} and depart from \cite{Bang-Jensen_Gutin2001}.  
Notice however that we set no cost, no capacity on the arcs and especially that we postulate acyclicity of the graph.} an acyclic, directed graph $(N,A)$ in which two special nodes are designated as the \textsl{source} $s$ and the \textsl{sink} $t$.  An \textsl{$s$--$t$ path} is a path starting at $s$ and ending at $t$.  

There are reasons to consider only acyclic networks $D$, rather than more general networks allowing for cycles.  First, the results often take an interesting, simpler form (also, we do not have the extensions to general networks of all the results presented here). 
Second, in the applications we have in view, the network happens to be acyclic (as in Sections~\ref{SE_MCP_flows} and \ref{SE_other}).

Consider a network $D = (N,A,s,t)$ for the rest of the subsection. 
A \textsl{flow (of value~$1$)} of $D$ is a point\footnote{In the literature, flows are often denoted by the letter $\Phi$; we prefer to use the letter $x$ because we view flows as particular points in the space $\R^A$.  When writing the coordinate of the point $x$ w.r.t.\ an arc $(u,v)$, we abbreviate $x((u,v))$ into $x(u,v)$.} $x$ from $\mathbb{R}^{A}$, associating a nonnegative number $x(a)$ to each arc $a$ of the network, such that the outflow $x(\delta^{+}(v))$ equals the inflow $x(\delta^{-}(v))$ at each node $v$ distinct from the source $s$ and the sink $t$, and at the source~$s$ the outflow $x(\delta^{+}(s))$ equals $1$ plus the inflow $x(\delta^{-}(s))$.   
All flows of $D$ form a polytope in $\R^A$, because by their definition they are the solutions of the following system of affine (in)equalities on $\R^A$
\begin{equation}\label{EQ_conservation}
\left\{
\begin{array}{rcl@{\qquad}l}
x(\delta^{+}(v)) - x(\delta^{-}(v)) &= &0, &\forall v \in N \setminus \{s,t\},\\
x(\delta^{+}(s)) - x(\delta^{-}(s)) &= &1,\\
x(a) &\geqslant &0, &\forall a \in A,
\end{array}
\right.
\end{equation}
and they form a bounded set because for any flow $x$ and any $a$ in $A$ there holds $0 \le x(a) \le 1$ (the latter inequality follows for instance from Theorem~\ref{thm_flow_decomposition} below, or directly by proving, for any topological sort $L$ of the acyclic directed  graph $(N,A)$ and any node $w$ in $N$, that the sum of the  $x(u,v)$'s with $u >_L w \ge_L v$ equals $0$ or $1$---which is easily done by recurrence along the nodes $w$ in $L$).

\begin{definition}\label{de_flow_polytope}
The \textsl{(value\z $1$-) flow polytope} $\FFF(D)$ of a network $D$ consists of all flows of $D$, in other words of all points $x$ in $\R^A$ that satisfy the system in \eqref{EQ_conservation}.  The latter system\footnote{In Section~\ref{se_Facets} we will removed repeated inequalities from the canonical description. Note that the canonical description is an affine description, but not  necessarily one of minimum size (as shown by Example~\ref{ex_a_network}).}  is the \textsl{canonical (affine) description} of the flow polytope $\FFF(D)$.
\end{definition}

For any flow in $\FFF(D)$, the net inflow at $t$ equals $1$; in other words, the flow polytope moreover satisfies
\begin{equation}\label{EQ_at_t}
x(\delta^{+}(t)) - x(\delta^{-}(t)) \;=\;-1.
\end{equation}
This is derived from Equations~ \eqref{EQ_conservation} together with
\begin{equation}\label{EQ_sum_of_conservations}
\left( \sum_{v\in N}\; x(\delta^{+}(v)) \right)  - \left( \sum_{v\in N}\; x(\delta^{-}(v)) \right)
 \;= \;0.
\end{equation}
The latter equation holds because for any $a\in A$, 
the term $x(a)$ appears once in each of the two summations.

There can be superfluous inequalities in the canonical description of $\FFF(D)$.
If for some node $v$ we have $\delta^-(v)=\{(u,v)\}$ and $\delta^{+}(v)=\{(v,w)\}$, the conservation law at $v$ implies $x(u,v)=x(v,w)$ for any $x$ in $\FFF(D)$, and so we may keep only one of the two inequalities $x(u,v)\ge0$ and $x(v,w)\ge0$.  
Equation~\eqref{EQ_minimum} displays a minimum affine description of the polytope $\FFF(D)$.

The next statement is the particular case for acyclic networks of the Flow Decomposition Theorem due to \cite{Gallai1958} and \cite{Ford_Fulkerson1962} (see also, for instance, \citealp{Korte_Vygen2008}, page\z 169).     

\begin{theorem}\label{thm_flow_decomposition}
Consider a network $D=(N,A,s,t)$.  Any flow $x$ of $D$ equals a  convex combination of the characteristic vectors $\chi^P$ of the $s$--$t$ paths $P$ of $D$.
\end{theorem}

Because the converse of Theorem\z \ref{thm_flow_decomposition} also holds (as easily seen), and the $\chi^P$ are $0$--$1$ points, we derive a geometric reformulation.

\begin{theorem}\label{thm_geometric_reformulation}
For any network $D=(N,A,s,t)$, the vertices of the flow polytope $\FFF(D)$ are exactly the characteristic vectors $\chi^P$ of all the $s$--$t$ paths $P$ of $D$.
\end{theorem}

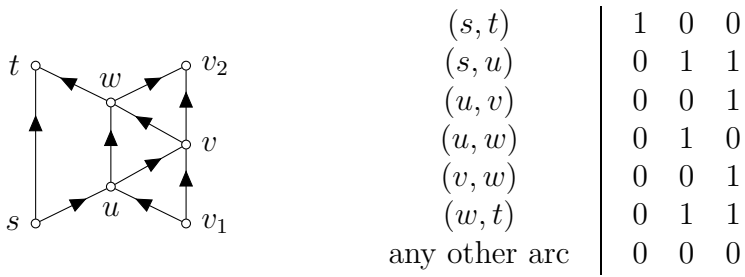
\begin{figure}[ht]
\begin{center}
~\hfill
\begin{tikzpicture}[xscale=1,yscale=0.7,baseline=10mm]
  \tikzstyle{vertex}=[circle,draw,fill=white, scale=0.3]
  \node (s) at (0,0) [vertex,label=left:$s$] {};
  \node (t) at (0,3) [vertex,label=left:$t$] {};
  \node (u) at (1,0.7) [vertex,label=below:{$u$}] {};
  \node (v) at (2,1.5) [vertex,label=right:{$v$}] {};
  \node (w) at (1,2.3) [vertex,label=above:{$w$}] {};
  \node (v1) at (2,0) [vertex,label=right:{$v_1$}] {};
  \node (v2) at (2,3) [vertex,label=right:{$v_2$}] {};

  \draw[->-=.7] (s) -- (t); \draw[->-=.7] (s) -- (u);
  \draw[->-=.7] (u) -- (v); \draw[->-=.7] (u) -- (w);
  \draw[->-=.7] (v) -- (w); 
  \draw[->-=.7] (v1) -- (u);
  \draw[->-=.7] (v1) -- (v);
  \draw[->-=.7] (w) -- (v2);
  \draw[->-=.7] (v) -- (v2);
  \draw[->-=.7] (w) -- (t);
\end{tikzpicture}
\hfill
$
\begin{array}{c@{\quad}|@{\quad}c@{\quad}c@{\quad}c}
(s,t) & 1 & 0 & 0\\
(s,u) & 0 & 1 & 1\\
(u,v) & 0 & 0 & 1\\
(u,w) & 0 & 1 & 0\\
(v,w) & 0 & 0 & 1\\
(w,t) & 0 & 1 & 1\\
\text{any other arc} & 0 & 0 & 0 
\end{array}
$
\hfill~
\end{center}
\caption{A network $D$ together with the ten coordinates (in columns) of the three vertices of the flow polytope $\FFF(D)$  (see Example~\ref{ex_a_network}).
\label{fig_first_example}
}
\end{figure}

\begin{example}\label{ex_a_network}
 Figure~\ref{fig_first_example} displays a  network $D$.  As $D$ has three $s$--$t$ paths, the flow polytope $\FFF(D)$ has three vertices (the characteristic vectors of the paths).  
The three columns contain the coordinates of the three vertices, respectively for the $s$--$t$ paths $(s,t)$, next  $(s,u), (u,w), (w,t)$, and finally $(s,u), (u,v), (v,w), (w,t)$.  
 The flow polytope $\FFF(D)$ is a convex  triangle lying in a space of dimension $10$.  Its canonical description is formed of six affine equalities and ten affine inequalities (so it is not a minimum-size affine description).
\end{example}

Many manuals on combinatorial optimization quote Theorem~\ref{thm_flow_decomposition}, 
which plays an important role in many applications.  
However, they do not say much on the geometric structure of the flow polytope $\mathcal F(D)$ of a network\z $D$.  
We collect in subsequent sections some related information.

Note that for each arc $a$ in $A$, the inequality $x(a)\ge0$ defines a face of the flow polytope $\mathcal F(D)$ (as explained in Subsection~\ref{SEC_Basic}), whose vertices are the (characteristic vectors of the) $s$--$t$ paths avoiding~$a$; the latter property will be often used in the sequel.
Proposition~\ref{PROP_FDI} characterizes the arcs for which the face is a facet.
  
There are many variants of the flow polytope $\mathcal F(D)$: when each arc of the network comes with a maximum capacity (see for instance \citealp{Korte_Vygen2008}); for flows not satisfying the conservation law \citep{Borgwardt_De-Loera_Finhold2018};
or under restrictions on the $s$--$t$ paths, \citealp{Stephan2009}; etc.

In the introduction, we mentioned that the MCP can be seen as a flow polytope.
This result, due to \cite{Fiorini2004}, is explained in the next section.  In Section~\ref{SE_other} we exhibit three other networks, whose flow polytopes play a role for the random utility models based on respectively weak orders, interval orders, and semiorders.  

\section{The Multiple Choice Polytope and Falmagne Theorem}
\label{SEC_MCP}

Let $\LLL\OOO_{\CCC}$ be the collection of all linear orderings of the alternative set $\CCC$.  Let moreover $\Lambda(\LLL\OOO_{\CCC})$ be the collection of all probability distributions on $\LLL\OOO_{\CCC}$.  We also set
\begin{equation}\label{Eq_E}
E \;:=\; \{\, (i,S) \st i \in S \in 2^\CCC \,\}.
\end{equation}

For each distribution $Pr$ in $\Lambda(\LLL\OOO_{\CCC})$, the Multiple Choice Model (MCM) predicts\footnote{We use classical terminology related to probabilistic models, see for instance \cite{Doignon_Heller_Stefanutti2018}.}
the various multiple choice probabilities $p(i,S)$ for $(i,S)\in E$ as
\begin{equation}\label{EQ_p_iS_repeated}
p(i,S) \;:=\; \sum\;\{\; Pr(L) \;\st\; L\in\LLL\OOO_\CCC \;\text{and}\;  \forall j \in S \setminus \{i\}:~i >_L j \;\}.
\end{equation}
We see the $p(i,S)$ as the coordinates of a point $p$ in $\R^E$. 
So the MCM is captured by the surjective mapping 
\begin{equation}\label{EQ_captured_by_f}
f:~\Lambda(\LLL\OOO_{\CCC}) \to \R^E:~ Pr \mapsto p.
\end{equation}
We extend $f$ to the mapping 
\begin{equation}\label{EQ_f_bar}
\bar f:~\R^{\LLL\OOO_{\CCC}} \to \R^E:~ t \mapsto p
\end{equation}
by setting for $(i,S)\in E$
\begin{equation}\label{EQ_t}
p(i,S) \;:=\; \sum\;\{\; t(L) \;\st\; L\in\LLL\OOO_\CCC \;\text{and}\;  \forall j \in S \setminus \{i\}:~i >_L j \;\}.
\end{equation}
Then $f$ is a linear mapping (each coordinate of $\bar f(t)$ is a sum of coordinates of $t$). 
The set of points predicted by the MCM is equal to $f(\Lambda(\LLL\OOO_{\CCC}))$, and also to $\bar f(\Lambda(\LLL\OOO_{\CCC}))$.  Because $\Lambda(\LLL\OOO_{\CCC})$ is a simplex and $\bar f$ is a linear mapping, the predicted points form a convex polytope, which we call the  \textsl{multiple choice polytope} (MCP) and denote as $\Pmc \CCC$.
In summary
\begin{equation}\label{EQ_summary_f_brown}
\begin{array}{c@{\qquad}c@{\qquad}c}
\R^{\LLL\OOO_{\CCC}} & \stackrel{\bar f}{\longrightarrow} &  \R^E \\
\cup & & \cup\\
\Lambda(\LLL\OOO_{\CCC}) & 
\stackrel{f}{\longrightarrow} &  \Pmc \CCC \\
\rotatebox[origin=c]{90}{$\in$} & & \rotatebox[origin=c]{90}{$\in$}\\
Pr & \stackrel{f}{\longmapsto}  & p
\end{array}
\end{equation}

Now for the probability distribution $Pr^L$ concentrated on the linear ordering $L$ of $\CCC$, denote by $p^L=f(Pr^L)$ the predicted point in $\Pmc \CCC$.  The various $Pr^L$ are the vertices of the simplex $\Lambda(\LLL\OOO_\CCC)$.  The image $f(Pr^L) = \bar f(Pr^L)$ is a point in $\R^E$, which we denote $p^L$. 
For $(i,S)\in E$, we have $p^L(i,S)$ equal to $1$ when $i >_L j$ for all $j \in S \setminus \{i\}$, and $0$ otherwise. 
The polytope $\Pmc \CCC$ is the convex hull of the images $p^L$ of the vertices $Pr^L$ of the simplex $\Lambda(\LLL\OOO_\CCC)$.  Because the images $p^L$ have coordinates $0$ or $1$, they are the vertices of $\Pmc \CCC$.

We reformulate the problem of characterizing the MCM as the problem of finding an affine description for the convex polytope MCP.  As we saw in the introduction, \cite{Falmagne1978} proves that the MCP is exactly the solution set of the system of (his generalized) Block Marschak inequalities.  Moreover, \cite{Fiorini2004} provides another proof of Falmagne Theorem by viewing the MCP as a flow polytope.  Let us explain this.

For $i\in\CCC$ and $L \in \LLL\OOO_{\CCC}$, the \textsl{beginning set} $L^-(i)$ and the \textsl{ending set $L(i)$}
are respectively
\begin{align}\label{Eq_infty}
L^-(i)       \;&:=\; \{j\in\CCC \st j \ge_L i \}\\
L(i)    \;&:=\; \{ j\in\CCC \st i \ge_L j \}.
\end{align}
In the present paragraph, we consider a fixed distribution $Pr$ on $\LLL\OOO_{\CCC}$, predicting the point $p=f(Pr)$ in $\Pmc \CCC$.
We moreover define for $i \in T \in 2^\CCC$
\begin{equation}\label{EQ_q_iT}
q(i,T) \;:=\; \sum\;\{\; Pr(L) \;\st\; L\in\LLL\OOO_\CCC \;\text{and}\;  T =  L(i)   \;\}. 
\end{equation}
Because if $i$ is ranked first in $S$ in some linear order $L$ there is only one superset $T$ of $S$ with $T =  L(i)$, there holds
\begin{equation}\label{EQ_p_from_q}
p(i,S) \;=\; \sum_{T\in 2^\CCC:\; T \supseteq S} q(i,T).
\end{equation} 
There follows from previous equation
\begin{equation}\label{EQ_Mobius}
q(i,T) \;=\; \sum_{S\in 2^\CCC:\; S \supseteq T} (-1)^{|S\setminus T|} \; p(i,S),
\end{equation}
by an application of the M\"obius inversion to the partially ordered set $(\{S\in\CCC\st i\in S\},\subseteq)$ (see for example \citealp{vanLint_Wilson2001}).  
By its definition in Equation~\eqref{EQ_q_iT}, $q(i,T)$ is nonnegative on $\Pmc \CCC$; therefore for all pairs  $(i,T)$ in $E$ and $p$ in $\Pmc \CCC$
\begin{equation}\label{EQ_BM}
\sum_{S\in 2^\CCC:\; S \supseteq T} (-1)^{|S\setminus T|} \; p(i,S) \;\ge\; 0.
\end{equation}
For $|T|=2$, \cite{Block_Marschak1960} prove that the last inequality  
holds for the MCM, and \cite{Falmagne1978} extends the result to all $T$'s.
Just above, we followed \cite{Fiorini2004} to derive the validity of \eqref{EQ_BM} for $\Pmc \CCC$.
Falmagne Theorem states that the system on $\R^E$ formed by all these affine inequalities, for $(i,T)\in E$, together with the obvious equations for $S$ in $2^\CCC$
\begin{equation}\label{EQ_obvious}
\sum_{i \in S} p(i,S) \;=\ 1
\end{equation} 
has $\Pmc \CCC$ as solution set.
 Next comes a summary of Fiorini's proof.

Consider the network $\Dlo \CCC = (2^\CCC,\prec,\es,\CCC)$ where the nodes are the subsets of $\CCC$, the arcs are the covering pairs of the inclusion relation on $2^\CCC$ (that is, all pairs $(T\setminus\{i\},T)$ for $i \in T \in 2^\CCC$), the source is the empty set $\es$, and the sink is $\CCC$.  Denote by $\Flo \CCC$ the flow polytope of the network $\Dlo \CCC$,
which lies in the space $\R^A$ for $A=\prec$.
Define now a mapping $\rho$ by
\begin{equation}\label{EQ_f}
\rho:~ \R^E \to \R^A:~p \mapsto r,
\end{equation}
where for $(T\setminus\{i\},T)$ in $A$ we set 
\begin{equation}\label{EQ_r_q}
r(T\setminus\{i\},T) \;:=\;  q(i,T)
\end{equation} 
with $q(i,T)$ as in \eqref{EQ_Mobius}.  Note that $\rho$ is a linear mapping (each coordinate of $\rho(p)$ is a linear combination of coordinates of $p$).  Moreover, $\rho$ has an inverse equal to the mapping
\begin{equation}\label{EQ_g}
\sigma:~ \R^A \to \R^E:~r \mapsto p,
\end{equation}
with $p(i,S)$ given by a rewriting of \eqref{EQ_p_from_q}:
\begin{equation}\label{EQ_p_r}
p(i,S) \;=\; \sum_{T\in 2^\CCC:\; S \subseteq T} r(T\setminus\{i\},T). 
\end{equation}

The mapping $\rho$ induces a bijection from the vertices of the multiple choice polytope $\Pmc \CCC$ to the vertices of the flow polytope $\Flo \CCC$: for any order $L$ with 
\begin{equation}\label{EQ_>_L}
i_1 \quad>_L\quad i_2 \quad>_L\quad \dots \quad>_L\quad i_n
\end{equation}
$\rho$ maps the vertex $p^L$ of $\Pmc \CCC$ onto the vertex of $\Flo \CCC$ which is the characteristic vector of the $s$--$t$
path 
\begin{equation}\label{EQ_path_for_L}
(\es,\{i_1\}),\quad (\{i_1\},\{i_1,i_2\}),\quad \dots,\quad (\{i_1,i_2,\dots,i_{n-1}\},\CCC)
\end{equation}
(so the beginning sets of $L$ are the nodes on the $\es$--$\CCC$ path, in the same order).
Consequently, the invertible linear mapping $\rho$ from $\R^E$ to $\R^A$ (where $A=\prec$) transforms the multiple choice polytope $\Pmc \CCC$ into the flow polytope $\Flo \CCC$.  Falmagne Theorem now follows at once  from Theorem~\ref{thm_geometric_reformulation}\footnote{\cite{Fiorini2004} rather refers to the total unimodularity of a certain matrix.} for the particular network $(2^\CCC,\prec,\es,\CCC)$.

\cite{Fiorini2004} proof shows the interest of flow polytopes to solve formal problems appearing in mathematical psychology.  More flow polytopes play a central role in \cite{Davis-Stober_Doignon_Fiorini_Glineur_Regenwetter2018}
(see our Section~\ref{SE_other}).  
Very recently, flow polytopes make their apparition in theoretical economics papers: for instance, \cite{Turansick2022} uses them to analyze the identification of the multiple choice model.  Also, \cite{Chang_Narita_Saito2022} refers in a proof to the adjacency of vertices on the multiple choice polytope.
  
In the next section 
we characterize the adjacency on any flow polytope, thus covering the adjacency on the multiple choice polytope as a particular case.

\section{Adjacency of Vertices on a Flow Polytope}
\label{se_Adjacency_of_Vertices}
In this section and the next three ones, we consider the flow polytope $\mathcal F(D)$ of a network $D=(N,A,s,t)$.  We  may assume that $D$ has at least one $s$--$t$ path, because otherwise $\mathcal F(D)$ is empty.
A characterization of the adjacency of vertices on a flow polytope is the object of Proposition~\ref{PROP_path_adj} below.
By Theorem~\ref{thm_geometric_reformulation}, the vertices of $\FFF(D)$ are the characteristic vectors $\chi^P$ of the $s$--$t$ paths $P$ of $D$.

\begin{lemma}\label{lem_smallest_face}
Let $\chi^{P_1}$, $\chi^{P_2}$, \dots, $\chi^{P_k}$ be vertices of the flow polytope $\FFF(D)$, 
that is, the characteristic vectors of $s$--$t$ paths $P_1$, $P_2$, \dots, $P_k$ of the network $D$.  The vertices of the smallest face of $\FFF(D)$ containing $\chi^{P_1}$, $\chi^{P_2}$, \dots, $\chi^{P_k}$ are exactly the vertices $\chi^R$ for $R$ an $s$--$t$ path such that $R \subseteq P_1 \cup P_2 \cup \dots \cup P_k$.
\end{lemma}

\begin{proof}
Let $U:=P_1 \cup P_2 \cup \dots \cup P_k$, and $F$ be the face of $\FFF(D)$ defined by the inequality 
\begin{equation}\label{EQ_smallest_face}
\sum_{a\in  A\setminus U}
x(a) \;\ge\; 0.
\end{equation}

Any vertex of $\FFF(D)$ equals $\chi^P$ for some $s$--$t$ path $P$; this vertex $\chi^P$ belongs to $F$ if and only if $a \notin P$ for each $a \in A\setminus U$ (so that the coordinate $x(a)$ takes value $0$ at $\chi^P$), that is, if and only if $P \subseteq U$.  

It remains to prove that the face $F$ is the smallest face of $\FFF(D)$ containing $\chi^{P_1}$, $\chi^{P_2}$, \dots, $\chi^{P_k}$.  
%
%
Let $G$ be any facet of $\FFF(D)$; thus $G$ is defined by the inequality $x(b)\ge0$ for some arc $b$ of $D$.  If $G$ contains $\chi^{P_1}$, $\chi^{P_2}$, \dots, $\chi^{P_k}$ then $b \in A \setminus U$.  Therefore $F \subseteq G$ (because if \eqref{EQ_smallest_face} is satisfied with equality at some point $x$ of $\FFF(D)$, then $x(b)=0$).  Hence any facet containing $\chi^{P_1}$, $\chi^{P_2}$, \dots, $\chi^{P_k}$ includes $F$.  Thus $F$ is the smallest face of $\FFF(D)$ containing $\chi^{P_1}$, $\chi^{P_2}$, \dots, $\chi^{P_k}$.
 \end{proof}

\begin{proposition}\label{PROP_path_adj}
Let $P$ and $Q$ be two $s$-$t$ paths of a network $D=(N,A,s,t)$.  The vertices $\chi^P$ and $\chi^Q$ of $\mathcal F(D)$ are adjacent if and only if
\begin{quote}
\textrm{($*$)}~whenever $P$ and $Q$ pass through a common internal node $v$, then $P$ and $Q$ coincide either before $v$ or after $v$. \end{quote}
\end{proposition}

\begin{proof}
By Lemma\z \ref{lem_smallest_face}, a vertex $\chi^R$ of $\mathcal F(D)$ (for some $s$--$t$ path $R$) belongs to the smallest face containing $\chi^P$ and $\chi^Q$ if and only if $R \subseteq P \cup Q$. 

If $P$ and $Q$ do not satisfy~($*$) for some common internal node\z $v$, we form a walk $R$ from $s$ to $t$ by following $P$ from $s$ to $v$, next $Q$ from $v$ to $t$.  Because of acyclicity, $R$ must be an $s$--$t$ path, and so the vertex $\chi^R$ belongs to the smallest face containing $\chi^P$ and $\chi^Q$.  Because $\chi^R$ differs from both $\chi^P$ and $\chi^Q$, the two latter vertices are nonadjacent.  

Conversely, assume that ($*$) holds.  We prove that the smallest face of $\mathcal F(D)$ containing the vertices $\chi^P$ and $\chi^Q$ does not contain any further vertex.  
Proceeding by contradiction, assume such a third vertex $\chi^R$ does exist.  Then $R$ is an $s$--$t$ path such that $R \subseteq P \cup Q$ and $R\neq P, Q$. 

Now let $(u,u')$ be the first arc of $R$ which lies outside $P$ or outside $Q$.  Assume $(u,u')\notin Q$, and thus $(u,u')\in P$ (otherwise, exchange the notations $P$, $Q$).   Because $R\neq P$, there must be a first arc $(v,v')$ in $R$ after $(u,u')$ such that $(v,v')\notin P$.  So $(v,v')\in Q$  in view of $R \subseteq P \cup Q$.  Then the node $v$ shows that Condition\z ($*$) does not hold, a contradiction.
\end{proof}

\begin{remark}\label{REM_NP_MT}
In the notation of the second paragraph of the proof above, we can create a second $s$--$t$ path $S$ by following $Q$ from $s$ to $v$, next $P$ from $v$ to $t$.  We have then $(\chi^P + \chi^Q)/2 = (\chi^R + \chi^S)/2$ because the equality holds for each coordinate $x(a)$, where $a\in A$.
Consequently, the flow polytope\z $\mathcal F(D)$ is a \textsl{combinatorial polytope} in the sense of \cite{Naddef_Pulleyblank1981}: it is a $0/1$-polytope in which for any pair of nonadjacent vertices, there is another pair of vertices having the same midpoint as the first pair.

As a matter of fact, the last assertion follows also from \cite{Matsui_Tamura1995}.  Any flow polytope $\FFF(D)$ is an \textsl{equality constraint polytope}, that is, its set of vertices is the set of $0$--$1$ points satisfying a given system of affine equations (in our case, the equalities in the canonical description of $\FFF(D)$).  It is thus also a polytope satisfying Properties~A and B of Matsui and Tamura.  Consequently all the findings of Matsui and Tamura hold for $\FFF(D)$, for instance those about linear optimization, or the fact that $\FFF(D)$ is a combinatorial polytope.
However, the results we present on flow polytopes (in particular on the MCP) differ in that they refer to $s$--$t$ paths and thus require the networks from which the polytopes are built. 
\end{remark}

\begin{figure}[ht]
\begin{center}
\begin{tikzpicture}[scale=1]
  \tikzstyle{vertex}=[circle,draw,fill=white, scale=0.3]
  \node (s) at (0,0) [vertex,label=left:$s$] {};
  \node (t) at (0,8) [vertex,label=left:$t$] {};
  
  \node (u_1) at (-1,1) [vertex,label=left:{$u_1$}] {};
  \node (v_1) at ( 1,1) [vertex] {};
  \node (w_1) at ( 0,2) [vertex,label=left:{$w_1$}] {};

  \node (u_2) at (-1,3) [vertex,label=left:{$u_2$}] {};
  \node (v_2) at ( 1,3) [vertex] {};
  \node (w_2) at ( 0,4) [vertex,label=left:{$w_2$}] {};

  \node (w) at ( 0,6) [vertex,label=left:{$w_{d-1}$}] {};
  \node (u_d) at (-1,7) [vertex,label=left:{$u_d$}] {};
  \node (v_d) at ( 1,7) [vertex] {};

  \draw[->-=.7] (s)   -- (u_1);  \draw[->-=.7] (u_1) -- (w_1);
  \draw[->-=.7] (w_1) -- (u_2);  \draw[->-=.7] (u_2) -- (w_2); 
  \draw[->-=.7] (w) -- (u_d);  \draw[->-=.7] (u_d) -- (t);
  
  \node at  (0,5) {$\vdots$};
  
  \draw[->-=.7] (s)   -- (v_1);  \draw[->-=.7] (v_1) -- (w_1);
  \draw[->-=.7] (w_1) -- (v_2);  \draw[->-=.7] (v_2) -- (w_2); 
  \draw[->-=.7] (w) -- (v_d);  \draw[->-=.7] (v_d) -- (t);

\end{tikzpicture}
\end{center} 
\caption{\label{fig_towards_a_cube}A network for Example~\ref{EX_TOWARDS_A_CUBE}, for each a natural number $d$ with $d\ge1$.}
\end{figure}
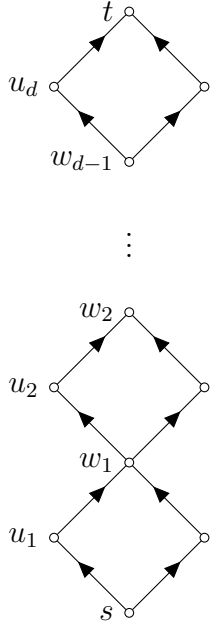

\begin{example}\label{EX_TOWARDS_A_CUBE}
For the network $D$ in Figure~\ref{fig_towards_a_cube}, it is an exercise to check that the flow polytpe $\FFF(D)$ is a $d$-dimensional $0/1$-cube (the vertices of $\FFF(D)$ are completely specified by the values, $0$ or $1$, of the coordinates $x(u_1,w_1)$, $x(u_2,w_2)$, \dots $x(u_{d-1},w_{d-1})$, and $x(u_d,t)$).  As announced in Remark~\ref{REM_NP_MT}, it is indeed a combinatorial polytope.  Moreover, the diameter of (the graph of ) the flow polytope equals~$d$.
\end{example}

\section{The Dimension of a Flow Polytope}
\label{se_Dimension}
Consider again the flow polytope $\mathcal F(D)$ of a network $D=(N,A,s,t)$, assuming that $D$ has at least one $s$--$t$ path.
Let $\widetilde  A$ denote the
subset of $A$ formed by all arcs of $D$ that belong to at least one $s$--$t$ path, and let $\widetilde  N$ be the subset of $N$ formed by all nodes of $D$ that appear on at least one arc in $\widetilde A$.
The network $\widetilde  D = (\widetilde  N, \widetilde  A, s, t)$ is called the \textsl{reduced network} of $D$, 
or the \textsl{reduction} of $D$ (for an illustration, see Figure~\ref{fig_corridors}).
For any node $u$ of $\tilde N$, denote with $\tilde\delta^-(u)$, resp.\ $\tilde\delta^+(u)$, the sets of arcs in $\tilde A$ with head, resp.\ tail $u$.
By Theorem\z \ref{thm_flow_decomposition}, the flow polytope $\mathcal F(D)$ satisfies $x(a)=0$ for any arc in $A \setminus \widetilde  A$.  Thus the flow polytopes $\FFF(\widetilde D)$ and $\mathcal F(D)$ are essentially the same polytope (they become equal when we naturally  assimilate  the space $\R^{\tilde A}$ with the linear subspace of the space $\R^A$ specified by  $x(a)=0$ for all $a\in A \setminus \widetilde  A$).
A network $D$ is \textsl{reduced} if $D=\widetilde  D$. 

\begin{figure}[ht]
\begin{center}
~\hfill
\begin{tikzpicture}[xscale=1,yscale=0.7]
  \tikzstyle{vertex}=[circle,draw,fill=white, scale=0.3]
  \node (s) at (0,0) [vertex,label=left:$s$] {};
  \node (t) at (0,3) [vertex,label=left:$t$] {};
  \node (u) at (1,0.7) [vertex,label=below:{$u$}] {};
  \node (v) at (2,1.5) [vertex,label=right:{$v$}] {};
  \node (w) at (1,2.3) [vertex,label=above:{$w$}] {};
  \node (v1) at (2,0) [vertex,label=right:{$v_1$}] {};
  \node (v2) at (2,3) [vertex,label=right:{$v_2$}] {};

  \draw[->-=.7] (s) -- (t); \draw[->-=.7] (s) -- (u);
  \draw[->-=.7] (u) -- (v); \draw[->-=.7] (u) -- (w);
  \draw[->-=.7] (v) -- (w); 
  \draw[->-=.7] (v1) -- (u);
  \draw[->-=.7] (v1) -- (v);
  \draw[->-=.7] (w) -- (v2);
  \draw[->-=.7] (v) -- (v2);
  \draw[->-=.7] (w) -- (t);
\end{tikzpicture}
\hfill
\begin{tikzpicture}[xscale=1,yscale=0.7]
  \tikzstyle{vertex}=[circle,draw,fill=white, scale=0.3]
  \node (s) at (0,0) [vertex,label=left:$s$] {};
  \node (t) at (0,3) [vertex,label=left:$t$] {};
  \node (u) at (1,0.7) [vertex,label=below:{$u$}] {};
  \node (v) at (2,1.5) [vertex,label=right:{$v$}] {};
  \node (w) at (1,2.3) [vertex,label=above:{$w$}] {};

  \draw[->-=.7] (s) -- (t); \draw[->-=.7] (s) -- (u);
  \draw[->-=.7] (u) -- (v); \draw[->-=.7] (u) -- (w);
  \draw[->-=.7] (v) -- (w); 
  \draw[->-=.7] (w) -- (t);
  
  \draw[->-=.7] (s) -- (t); \draw[->-=.7] (s) -- (u);
  \draw[->-=.7] (u) -- (v); \draw[->-=.7] (u) -- (w);
  \draw[->-=.7] (v) -- (w); 
  \draw[->-=.7] (w) -- (t);
\end{tikzpicture}
\hfill~
\end{center}
\caption{On the left, a nonreduced network; on the right, its reduction.
\label{fig_corridors}
}
\end{figure}
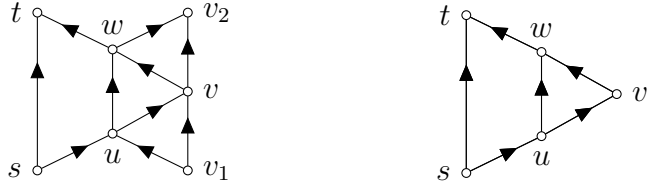
 
\begin{proposition}\label{prop_dim}
Suppose the network $D = (N,A,s,t)$ has at least one $s$--$t$ path, and let $\widetilde D = (\widetilde N,\widetilde A,s,t)$ be its reduced network.  Then the dimension of the flow polytope $\mathcal F(D)$  equals $|\widetilde  A| - |\widetilde  N| + 1$. 
\end{proposition}

\begin{proof}  
As we saw in the paragraph before the statement we may 
assimilate $\mathcal F(D)$ with $\FFF(\widetilde  D)$, a polytope lying in $R^{\widetilde A}$.  
By definition, $\FFF(\widetilde  D)$ is the solution set of the system on $R^{\widetilde  A}$
\begin{equation}\label{EQ_conservation_tilde}
\left\{\begin{array}{rcl@{\quad}l}
x(\widetilde \delta^{+}(v)) - x(\widetilde \delta^{-}(v)) &= &0, &\forall v \in \widetilde  N \setminus \{s,t\},\\
x(\widetilde \delta^{+}(s)) - x(\widetilde \delta^{-}(s)) &= &1,\\
x(a) &\geqslant &0, &\forall a \in \widetilde  A.
\end{array}
\right.
\end{equation}
Hence $\FFF(\widetilde D)$ lies in the subspace of $\R^{\widetilde A}$ defined by the $|\widetilde N|-1$ affine equations in \eqref{EQ_conservation_tilde}.  We first show that the subspace has dimension at most $|\widetilde A| - (|\widetilde N|-1)$ by establishing that the $|\widetilde N|-1$ affine equations are independent.
It suffices to exhibit, for each of the equalities in \eqref{EQ_conservation_tilde}, a point in $\R^{\widetilde A}$ which satisfies all equalities in \eqref{EQ_conservation_tilde} but the one considered.  
Let first $v$ be a node in $\widetilde A\setminus\{s,t\}$.  Take any path $U$ 
in $(\widetilde N, \widetilde A,s,t)$ from $s$ to $v$ (such a path exists because $v$ is on some $s$--$t$ path).  The characteristic vector $\chi^U$ satisfies all inequalities in \eqref{EQ_conservation_tilde} as well as all equalities but the one for $v$.  
Second, assume $v=s$.  The null vector in $\R^{\widetilde A}$ does the job. 

From previous paragraph $\dim \mathcal \FFF(\widetilde D) \le |\widetilde A| - (|\widetilde N|-1)$.  To prove the opposite inequality, we show the existence of $1 + |\widetilde A| - (|\widetilde N|-1) $ affinely independent vertices in $F(\widetilde D)$
(Remark\z \ref{rem_alternate_argument} below provides an alternate argument).
Because the reduced network $\widetilde D$ is acyclic, it admits a  topological sort $L$ of its nodes, say 
\begin{equation}
u_1 \quad>_L\quad u_2 \quad>_L\quad  \dots \quad>_L\quad u_m,
\end{equation}
with $u >_L v$ for any arc $(u,v)$ in $\widetilde A$ and 
$m=|\widetilde N|$ (necessarily $u_1=s$ and $u_m=t$ in view of the definition of $\widetilde D$). 
Now for each node $u$ distinct from $u_1$, paint in green one arbitrarily chosen arc in $\widetilde A$ with head $u$.  Thus $|\widetilde N| - 1$ arcs were just painted in green; paint in blue all the other arcs. 

Form a first $s$--$t$ path $P_G$ using only green arcs.  This path is uniquely determined: its last arc is the green arc $(u_k,u_m)$ with head $u_m$ (for some unique  $k$), the arc before $(u_k,u_m)$ is the green arc with head $u_k$, etc.  

Next, for any of the $|\widetilde A| - (|\widetilde N|-1)$ blue arcs, say $(u,v)$, form an $s$--$t$ path by first following green arcs from $s$ to $u$ (there is only one suitable sequence of green arcs), next 
follow the blue arc$(u,v)$ and finally arcs (green or blue) from $v$ to $t$ (such arcs do exist because $v$ is on some $s$--$t$ path).
The characteristic vectors of the resulting $s$--$t$ paths, in number $1 + |\widetilde A| - (|\widetilde N| - 1)$, are affinely independent, as we next show.  

Build as follows a list $M$ of the $|\widetilde A| - |\widetilde N| + 2$ $s$--$t$ paths we just constructed:  $M$ collects first, in any order, all the $s$--$t$ paths formed for the blue arcs with tail $u_1$ (if any); next in any order the $s$--$t$ paths formed for the blue arcs with tail $u_2$ (if any); \dots; the $s$--$t$ paths formed for the blue arcs with tail $u_{m-1}$ if any; finally, 
the last item in the list $M$ is the $s$--$t$ path $P_G$ consisting only of green arcs.
Then the characteristic vector of any $s$--$t$ path $P$ in $M$ distinct from $P_G$ is affinely independent from the characteristic vectors of all the $s$--$t$ paths listed in $M$ after $P$.  Indeed, if $P$ was formed for the blue arc $(u,v)$, then $(u,v)$ belongs to $P$ but not to any of the $s$--$t$ paths listed after $P$ in $M$.  Thus the characteristic vector $\chi^P$ satisfies $x(u,v)\neq 0$ while all the characteristic vectors of the $s$--$t$ paths after $P$ in $M$ satisfy $x(u,v) = 0$.
\end{proof}

\begin{remark}\label{rem_alternate_argument}
The proof of the second inequality can be replaced with a call to Theorem\z 5.6 of \cite{Schrijver2003}.   Because no inequality $x(a)\ge0$, for $a\in \widetilde A$, is satisfied with equality by $\mathcal F(D)$, the dimension of $F(\widetilde  D)$ equals $|\widetilde  A|$ (the dimension of the space in which $F(\widetilde D)$ lies) minus the rank of the matrix of coefficients of the variables in the affine equations in \eqref{EQ_conservation_tilde}.  From the first half of the proof, we know that the rank equals $|\widetilde N|-1$.
\end{remark}

\section{The Facets of a Flow Polytope}
\label{se_Facets}
We now aim at recognizing the facets of the flow polytope $\mathcal F(D)$ of a network $D=(N,A,s,t)$.
In view of the canonical description of $\mathcal F(D)$ in\z \eqref{EQ_conservation}, any facet is for sure defined by an inequality $x(a)\ge0$ for some arc in $\widetilde  A$ (remember from Section~\ref{se_Dimension} that for $b \in A \setminus \widetilde  A$, the flow polytope $\mathcal F(D)$  satisfies $x(b)=0$).
Proposition\z \ref{PROP_FDI} below characterizes the arcs $a$ such that $x(a)\ge0$ defines a facet of $\mathcal F(D)$, referring to the notions of `corridors' and `good arcs' (see Example~\ref{ex_corridors} and Figure~\ref{fig_corridors} for an illustration).

For a node $u$ in the network $D=(N,A,s,t)$,
set $\widetilde d^-(u) = |\widetilde\delta^-(u)| $ and 
$\widetilde d^+(u) = |\widetilde\delta^+(u)|$.

\begin{definition}\label{DEF_corridor}
A \textsl{corridor} of the network $D$ is a path of the reduced network $\widetilde D = (\widetilde N, \widetilde  A, s, t)$
\begin{equation}\label{EQ_corridor}
(u_1,u_2), \quad (u_2,u_3), \quad  \dots,\quad (u_{m-1},u_m)
\end{equation} 
 such that
\begin{gather}
\widetilde d^-(u_2) = \widetilde d^+(u_2) =
\widetilde d^-(u_3) = \widetilde d^+(u_3) =
\dots = \widetilde d^-(u_{m-1}) = \widetilde d^+(u_{m-1}) = 1
\end{gather}
which is maximal (w.r.t.\ the inclusion of arc sets) for this property, that is
\begin{equation}\label{eq_maximality}
\big(\widetilde d^-(u_1) \neq1 \text{ or } \widetilde d^+(u_1)\neq1\big) 
\;\;\text{and}\;\; 
\big(\widetilde d^-(u_m) \neq1 \text{ or } \widetilde d^+(u_m) \neq1\big).
\end{equation}  
The corridor in \eqref{EQ_corridor} is \textsl{good} when $\widetilde d^+(u_1) \ge 2$ and $\widetilde d^-(u_m) \ge 2$.  An arc is \textsl{good} if it belongs to some good corridor.  We  call arcs or corridors \textsl{bad} if they are not good.
\end{definition}

\begin{example}\label{ex_corridors}
The network $D$
on the left in Figure~\ref{fig_corridors} is not reduced.  Its 
reduction $\widetilde D$ is on the right.
Both networks have three good corridors, namely
\begin{equation}
(s,t),\qquad (u,w),\qquad \text{and} \qquad (u,v), \quad (v,w),
\end{equation}
and two bad corridors, namely
\begin{equation}
(s,u)\qquad \text{and} \qquad (w,t).
\end{equation}
\end{example}

Definition\z \ref{DEF_corridor} implies that no arc in $A \setminus \widetilde A$ belongs to any corridor, while each arc $a$ in $\widetilde  A$ belongs to a unique corridor (sometime reduced to itself), which we denote as $\cor (a)$. 
Said otherwise, the corridors of the network $D=(N,A,s,t)$ form a partition of $\widetilde A$.  Moreover, if an $s$--$t$ path contains any arc of some corridor, then it includes the whole corridor.

For the corridor in \eqref{EQ_corridor}, the flow polytope satisfies
\begin{equation}\label{EQ_equalities}
x(u_1,u_2) \;=\; x(u_2,u_3) \;=\; \dots \;=\; x(u_{m-1},u_m)
\end{equation}
(because of the conservation law at nodes $u_2$, $u_3$, \dots, $u_{m-1}$).  
In the canonical description of $\FFF(D)$, from all the inequalities $x(u_{i-1},u_i)\ge0$ for $i=2$, $3$, \dots, $m$, we keep only one, namely $x(u_1,u_2)\ge0$.

\begin{lemma}\label{lem_degree_1}
Let $D = (N,A,s,t)$ be a network, and $(u,v)$ be an arc in $\widetilde A$ satisfying at least one of the two following conditions: 
\begin{enumerate}[\quad\rm(i)~]
\item  $\widetilde d^-(u) \neq1   \quad\text{and}\quad 
      \widetilde d^+(u)=1$;
\item  $\widetilde d^-(v) = 1      \quad\text{and}\quad \widetilde d^+(v)
 \neq 1$.
\end{enumerate}
Then the face $F$ of the flow polytope $\mathcal F(D)$ defined by the inequality $x(u,v)\ge0$ cannot be a facet of $\mathcal F(D)$.
\end{lemma} 

\begin{proof}
We consider only Assumption~(ii), the proof under Assumption~(i) being similar.
A priori, there are
three cases for $v$.

If $v=t$, then we have for each point $x$ of $\mathcal F(D)$ (because the net inflow at\z $t$ equals\z $1$, see Equation\z \eqref{EQ_at_t})
\begin{equation}
x(u,v) \;=\; 1 + \sum\{x(t,w) \st (t,w)\in \delta^+(t)\}.
\end{equation}
Even if there is no term in the summation, the last equation implies that $x(u,v)=0$ is impossible, so $F$ is the empty face.  For the empty set to be a facet of $\mathcal F(D)$, it must be that $D$ has a single $s$--$t$ path.
This contradicts (ii).

The case $v=s$ is impossible because of the acyclicity of $D$ (remember that $(u,v)\in\widetilde A$ means that $(u,v)$ belongs to some $s$--$t$ path).

Letting now $v\neq s$, $t$, we prove that $F$ cannot be a facet.
From the present assumptions $(u,v)\in\widetilde A$, $v \neq t$, and $\widetilde d^+(v) \neq 1$, we derive $\widetilde d^+(v) \ge 2$.  For any flow $x$ in $\mathcal F(D)$, the conservation law at $v$ gives 
\begin{equation}\label{eq_in_proof}
x(u,v) \;=\; \sum\{x(v,w) \st (v,w)\in\widetilde \delta^+(v)\}.
\end{equation} 
Hence $x(u,v)=0$ if and only if $x(v,w)=0$ for all $(v,w) \in \widetilde \delta^+(v)$.  Thus the face defined by $x(u,v)\ge0$ is the intersection of the faces defined by $x(v,w)\ge0$, for $(v,w)\in\widetilde \delta^+(v)$,  each of the latter faces being proper because $\delta^+(v) \subseteq \widetilde A$. 
Moreover, at least two such faces must differ because any $s$--$t$ path $P$ containing $(u,v)$ contains exactly one arc $(v,w)$ in $\widetilde \delta^+(v)$, hence the vertex $\chi^P$ satisfies $x(v,w)\neq0$ and also $x(v,w')=0$ for $(v,w')\in \widetilde \delta^+(v)\setminus\{(v,w)\}$.  
We conclude that $F$ cannot be a facet.
\end{proof}

\begin{lemma}\label{lem_AequalB}
Let $D=(N,A,s,t)$ be a network. For the two arcs 
$a$ and $b$ of $\widetilde A$, assume that both inequalities $x(a)\ge0$ and $x(b)\ge0$ on $\R^A$ define facets $F_a$ and $F_b$ of $\FFF(D)$ respectively.
Then $F_a=F_b$ if and only if $a$ and $b$ belong to the same corridor.
\end{lemma}

\begin{proof}
If $\cor(a)=\cor(b)$, then $x(a)=x(b)$ for any flow $x$ in $\FFF(D)$ and so $F_a=F_b$.  

To prove the converse, assume $F_a=F_b$.   Because an empty polytope has no facet, $D$ must have at least one $s$--$t$ path. 
If $D$ has a single $s$--$t$ path, $a$ and $b$ belong for sure to the unique corridor of $D$.  Assume from now on that $D$ has at least two $s$--$t$ paths. 
There exists some $s$--$t$ path $P$ containing the arc $a$ (because the facet $F_a$ must exclude some vertex of $\mathcal F(D)$).  Because $F_a$ and $F_b$ avoid exactly the same vertices, $P$ must also contain $b$; say that $a$ comes before $b$ in $P$ (otherwise relabel $a$ and $b$).  Now $\cor(a)$ and $\cor(b)$ are subsets of $P$.   If they differ, we derive a contradiction as follows.  The last node $v$ on $\cor(a)$ must then come along $P$ before $\cor(b)$ (here $v$ can be the head of $a$ and/or the tail of $b$).   We have $\widetilde d^-(v) \ge 2$ or $\widetilde d^+(v) \ge 2$. 

If $\widetilde d^-(v) \ge 2$, there exists some arc $(u,v)$ in $\widetilde \delta^-(v)$ not in $\cor(a)$.  The arc $(u,v)$ is in some $s$--$t$ path $Q$.  Following $Q$ from $s$ to $v$, and next $P$ from $v$ to $t$, we get an $s$--$t$ path $R$ (in view of the acyclicity of $D$).  As $R$ excludes the arc $a$ but contains the arc $b$, the vertex $\chi^R$ is in $F_a$ but not in $F_b$, a contradiction.

If $\widetilde d^-(v) < 2$, then $\widetilde d^-(v) = 1$ and $\widetilde d^+(v) \ge 2$.  Let $u$ be this time the node preceding $v$ on $P$.   Then $x(u,v)\ge0$ also defines the facet $F_a$ (because the arcs $(u,v)$ and $a$ belong to the same corridor).  By Lemma\z \ref{lem_degree_1}(ii), $F_a$ cannot be a facet, a contradiction.
\end{proof}

\begin{remark}\label{rem_one_network}
In the proof of sufficiency in Lemma\z \ref{lem_AequalB}
(from right to left) we do not need the assumption that $F_a$ and $F_b$ are facets, faces is enough.
To the contrary, the necessity part (left to right) of Lemma\z \ref{lem_AequalB} does not remain true if we replace `facet' by `face' in the statement.  This is shown by the arcs $(s,u)$ and $(w,t)$ in the network $D$ displayed in Figure\z \ref{fig_corridors}.  Here the flow polytope $\mathcal F(D)$ has three vertices.  Its three facets are respectively defined by the inequalities $x(s,t)\ge0$, $x(u,w)\ge0$, $x(u,v)\ge0$ (or $x(v,w)\ge0$). 
Both inequalities $x(s,u)\ge0$ and $x(w,t)\ge0$ define the same $0$-dimensional face; however, they are in distinct corridors. 
\end{remark}

\begin{proposition}\label{PROP_FDI}
Given an arc $a$ in the network $D= (N,A,s,t)$, the inequality $x(a) \ge 0$ defines a facet of the flow polytope $\mathcal F(D)$ if and only if the arc $a$ belongs to $\widetilde A$ and moreover either the network $D$ has a single $s$--$t$ path, or the arc\z $a$ is good.
\end{proposition} 

\begin{proof} 
When $a$ belongs to some $s$--$t$ path, we assume that the successive arcs in $\cor(a)$ (the corridor containing $a$) are
\begin{equation}\label{EQ_corridor_bis}
(u_1,u_2), \quad (u_2,u_3),\quad \dots,\quad (u_{m-1},u_m).
\end{equation} 
For all arcs $b$ in $\cor(a)$ the polytope $\mathcal F(D)$ satisfies $x(a)=x(b)$ (as in \eqref{EQ_equalities}).  Therefore, in the canonical description of $\FFF(D)$, we keep only one of the inequalities $x(b)\ge0$ for $b\in\cor(a)$, namely $x(a)\ge0$.

To prove sufficiency, first note that if $D$ has a single $s$--$t$ path, then $\mathcal F(D)$ has only one point and moreover $x(a)\ge0$ defines the empty facet, which is here a facet of $\mathcal F(D)$.
Now suppose that the arc $a$ is good, which in the notation of \eqref{EQ_corridor_bis} means $\widetilde d^+(u_1) \ge 2$ and $\widetilde d^-(u_m) \ge 2$.  To show that the inequality $x(a) \ge 0$ defines a facet, it suffices to exhibit some point $y$ of $\mathbb{R}^{A}$ that satisfies all the affine equations and inequalities of the canonical description of $\mathcal F(D)$ except for the inequality $x(a) \ge 0$.  
Take some arc $(u,u_m)$ in $\widetilde \delta^-(u_m)\setminus\{(u_{m-1},u_m)\}$.  Thus there exists some $s$--$t$ path containing $(u,u_m)$, and so also a path $M$ starting at $s$ with last arc $(u,u_m)$. 
Now take some arc $(u_1,v)$ in $\widetilde \delta^+(u_1)\setminus\{(u_1,u_2)\}$.  There exists some $s$--$t$ path containing $(u_1,v)$, and so a path $P$ with first arc $(u_1,v)$ and ending at $t$.  Set $C:=\cor(a)$.  The point $y=\chi^M + \chi^P - \chi^C$ in $\mathbb{R}^{A}$ has the desired property (even if $M$ and $P$ pass through some common nodes and/or share some arcs). 

To prove necessity, assume that the inequality $x(a)\ge0$ defines a facet.  First note that $a$ must belong to some $s$--$t$ path otherwise the facet defined by $x(a) \ge 0$ would contain all vertices of $\mathcal F(D)$.   Hence $a\in \widetilde A$.
Assume further that the arc $a$ is bad.  Then for its corridor $\cor(a)$ written as in \eqref{EQ_corridor_bis}, there holds $\widetilde d^+(u_1) = 1$ or $\widetilde d^-(u_m) = 1$.  In the first case, we must also have $\widetilde d^-(u_1) \neq 1$ (by \eqref{eq_maximality}), and so a contradiction follows from Lemma\z \ref{lem_degree_1}(i).
In the second case, we have $\widetilde d^+(u_m) \neq 1$, and a contradiction follows from Lemma\z \ref{lem_degree_1}(ii).
\end{proof}

\begin{corollary}\label{COR_number_facets}
The number of facets of the flow polytope $\mathcal F(D)$ of a network $D$ equals the number of good corridors of $D$.
\end{corollary}

\begin{proof}
This follows at once from Proposition\z \ref{PROP_FDI} and Lemma\z \ref{lem_AequalB}.    
\end{proof}

From Proposition~\ref{prop_dim} and the proof of Proposition~\ref{PROP_FDI} we derive a mini\-mum-size affine description of $\FFF(D)$.  Let $B$ be a subset of $A$ which is a \textsl{transversal} of the collection of corridors, that is, $B$ contains exactly one arc from each corridor.  The system
\begin{equation}\label{EQ_minimum}
\left\{\begin{array}{rcl@{\quad}l}
x(a) &=& 0, &\forall a \in A \setminus\widetilde  A,\\
x(\widetilde \delta^{+}(v)) - x(\widetilde \delta^{-}(v)) &= &0, &\forall v \in \widetilde  N \setminus \{s,t\},\\
x(\widetilde \delta^{+}(s)) - x(\widetilde \delta^{-}(s)) &= &1,\\
x(b) &\geqslant &0, &\forall b \in B
\end{array}
\right.
\end{equation}
is an affine description of $\FFF(D)$ having minimum size.

\section{The Adjacency of Facets of a Flow Polytope}
\label{se_Adjacency_of_Facets}

By definition, two facets of a polytope are \textsl{adjacent} if their intersection 
is a face of dimension equal to the dimension of the polytope minus $2$. 
See Figure~\ref{FIG_gray_area} for an illustration of the next characterization of (non-)adjacency of facets of a flow polytope.

\def\ratiox{1}  
\def\ratioy{1} 
\begin{figure}[ht]
\begin{center}
~\hfill 
\begin{tikzpicture}[xscale=\ratiox,yscale=\ratioy]
\tikzstyle{vertex}=[circle,draw,fill=white, scale=0.3]

\node[vertex,label=left:$v~$] (v) at (0,0)  {};
\node[vertex] (ah) at (-1,1.5)  {};
\node[vertex] (bh) at (1.5,1.7){};

\draw[->-=.6,double] (v) to node[left]{$\cor(a)~$} (ah);
\draw[->-=.6,double] (v) to node[right]{$~\cor(b)$} (bh);

\draw[->-=.6] (-0.5,-1) node[vertex,label=left:$u~$]{} to (v);
\draw[->-=.6] ( 0.5,-1) node[vertex,label=right:$~u'$]{} to (v);

\draw[ultra thick] ($(v) +(-2mm,0)$) arc (180:0:2mm);
\end{tikzpicture}
\hfill 
\begin{tikzpicture}[xscale=\ratiox,yscale=\ratioy]
\tikzstyle{vertex}=[circle,draw,fill=white, scale=0.3]

\node[vertex,label=left:$v~$] (v) at (0,0)  {};
\node[vertex] (u) at (0,-1.5)  {};
\node[vertex] (ah) at (-1,1.5)  {};
\node[vertex] (bh) at (1.5,1.7){};

\draw[->-=.6,double] (v) to node[left]{$\cor(a)~$} (ah);
\draw[->-=.6,double] (v) to node[right]{$~\cor(b)$} (bh);

\draw[->-=.6,double] (u) to node[left]{$\cor(u,v)~$} (v) ;
\draw[->-=.7] (-0.5,-2.5)node[vertex,label=left:$~$]{} to (u);
\draw[->-=.7] ( 0.5,-2.5)node[vertex,label=right:$~$]{} to (u);

\draw[ultra thick] ($(v) +(-2mm,0)$) arc (180:0:2mm);
\end{tikzpicture}
\hfill~
\end{center}

\kern2mm

\begin{center}
~\hfill
\begin{tikzpicture}[xscale=\ratiox,yscale=\ratioy]
  \tikzstyle{vertex}=[circle,draw,fill=white, scale=0.3]

\node[vertex,label=left:$u~$] (u) at (0,0)  {};
\coordinate (v) at (0,1){};
\textbf{}
\node[vertex] (at) at (-1,-1.5)  {};
\node[vertex] (bt) at (1.5,-1.7){};

\draw[->-=.6,double] (at) to node[left]{$\cor(a)~$} (u);
\draw[->-=.6,double] (bt) to node[right]{$~\cor(b)$} (u);

\draw[->-=.6] (u) to (-0.5,1) node[vertex,label=left:$v~$]{};
\draw[->-=.6] (u) to ( 0.5,1) node[vertex,label=right:$~v'$]{};

\draw[ultra thick] ($(u) +(-2mm,0)$) arc (180:360:2mm);
\end{tikzpicture}
\hfill 
\begin{tikzpicture}[xscale=\ratiox,yscale=\ratioy]
  \tikzstyle{vertex}=[circle,draw,fill=white, scale=0.3]

\node[vertex,label=left:$u~$] (u) at (0,0)  {};
\node[vertex] (v) at (0,1.5){};
\textbf{}
\node[vertex] (at) at (-1,-1.5)  {};
\node[vertex] (bt) at (1.5,-1.7){};

\draw[->-=.6,double] (at) to node[left]{$\cor(a)~$} (u);
\draw[->-=.6,double] (bt) to node[right]{$~\cor(b)$} (u);

\draw[->-=.6,double] (u) tonode[left]{$\cor(u,v)~$} (v);

\draw[->-=.6] (v) to (-0.5,2.5) node[vertex,label=left:$~$]{};
\draw[->-=.6] (v) to ( 0.5,2.5)node[vertex,label=right:$~$]{};

\draw[ultra thick] ($(u) +(-2mm,0)$) arc (180:360:2mm);
\end{tikzpicture}
\hfill~
\end{center}
\caption{
\label{FIG_gray_area} 
An illustration of  Proposition~\ref{PRO_non_adjacency_of_facets}: on top, Condition~(i) with the half-circle indicating $\widetilde d^+(v)=2$; on bottom, Condition~(ii) with the half-circle indicating $\widetilde d^-(u)=2$.  
}
\end{figure}
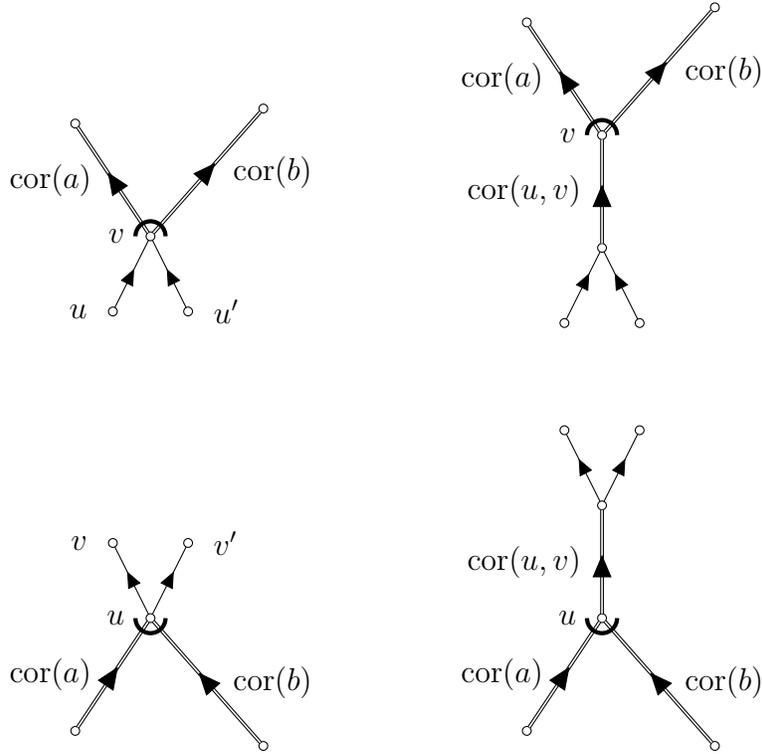

\begin{proposition}\label{PRO_non_adjacency_of_facets}
For two good arcs $a$ and $b$ in a network $D=(N,A,s,t)$, let $F_a$ and $F_b$ be the facets of the flow polytope $\FFF(D)$ respectively defined by $x(a)\ge0$ and $x(b)\ge0$.  The facets $F_a$ and $F_b$ are \underline{not} adjacent if and only if at least one of the two following conditions holds:
\begin{enumerate}[\qquad\rm(i)]
\item the corridors $\cor(a)$ and $\cor(b)$ have the same initial node, say $v$, with $\widetilde d^+(v)=2$, and 
\begin{enumerate}[\qquad\rm(1)]
\item either $\widetilde d^-(v)\ge2$,
\item or $\widetilde\delta^-(v)=\{(u,v)\}$ and the initial node of $\cor(u,v)$ has in-degree at least~$2$; 
\end{enumerate}
\item the corridors $\cor(a)$ and $\cor(b)$ have the same terminal node, say $u$, with $\widetilde d^-(u)=2$, and
\begin{enumerate}[\qquad\rm(1)]
\item either $\widetilde d^+(u)\ge2$,
\item or $\widetilde\delta^+(u)=\{(u,v)\}$ and the final node of $\cor(u,v)$ has out-degree at least~$2$.
\end{enumerate}
\end{enumerate}
\end{proposition}  

\begin{proof}
(Necessity).  For any polytope, two of its facets $F$ and $G$ are \underline{not} adjacent if and only if there exists some facet $K$ such that $F \cap G \subseteq K$ with $K$ distinct from $F$ and $G$. 
 
In view of Proposition~\ref{PROP_FDI}, nonadjacency of the given facets $F_a$ and $F_b$ of $\FFF(D)$ implies the existence of some good arc $c$ for which the facet $F_c$ defined by the inequality $x(c)\ge0$ includes $F_a \cap F_b $ and is distinct from $F_a$ and $F_b$ (note that $F_a\neq F_b$ implies that the network has more than one $s$--$t$ path). 
Then by Lemma~\ref{lem_AequalB} $\cor(c) \neq \cor(a)$, $\cor(c) \neq\cor(b)$.  All vertices of the face $F_a \cap F_b$ are vertices of $F_c$, equivalently all $s$--$t$ paths containing $c$ also contain $a$ or $b$.

Take some $s$--$t$ path $P$ containing $c$ (there exists such a $P$ because $F_c\neq\FFF(D)$).
Say that $P$ contains $a$ (if $P$ does not contain $a$, exchange the notation $a$ and $b$), then $P$ includes $\cor(a)$.  In $P$, the arc $a$ comes either after the arc $c$ or before $c$.  Treating only the second case, we will derive (ii) (in a similar way, the first case leads to (i)). 

Let $u$ be the final node of $\cor(a)$, and $v$ be the final node of $\cor(b)$.  We first prove $u=v$.
Because $a$ is good, there exists some arc $(u',u)$ in $\tilde A$ outside $\cor(a)$, thus also outside $P$. 
Take an $s$--$t$ path $Q$ containing the arc $(u',u)$.
Following $Q$ from  $s$ to $u$, next $P$ from $u$ to $t$ we get an $s$--$t$ path $R$ containing $c$ which avoids $a$ and passes through $u$.  
Then $R$ must contain $b$, thus $R$ includes $\cor(b)$.
Now if $u\neq v$, we derive a contradiction in each of the two remaining possible positions of $v$ in $R$ with respect to $\cor(c)$:
\begin{enumerate}[\rm(i)~]
\item $v$ comes in $R$ after the last node of $\cor(c)$.  Then the initial node $v_1$ of $\cor(b)$ comes on $R$ at or after the last node of $\cor(c)$.  Because the 
arc~$b$
is good, there is some arc $(v_1,w)$ in $\widetilde\delta^+(v_1)$ outside $\cor(b)$. Following $R$ from $s$ to $v_1$, next $(v_1,w)$, finally some path from $w$ to $t$, we obtain an $s$--$t$ path containing $c$ but neither $a$ nor $b$, a contradiction.
\item 
$v$ comes in $R$ before or at the initial node of $\cor(c)$.
Because the arc $b$ is good, there is an arc 
$(v',v)$ outside $\cor(b)$, thus an $s$--$t$ path containing $(v',v)$.  Following this last path from $s$ to $v$, next $R$ from $v$ to $t$, we get an $s$--$t$ path $S$ containing $c$ but not $b$.  If $S$ happens to avoid $a$, we have a contradiction. If $S$ contains $a$, then $a$ must be before $b$ on $S$ and 
we can then similarly build an $s$--$t$ path $S$ containing $c$ but neither $b$ nor $a$, the same contradiction.
%
\end{enumerate}

We have thus proved $u=v$.
In view of $\cor(a)\neq\cor(b)$, there holds $\widetilde d^-(u)\ge2$.    
If $\widetilde d^-(u)>2$ were true, there would exist some arc $(w,u)$ outside $\cor(a) \cup \cor(b)$.  Following some $s$--$t$ path from $s$ to $w$, next $(w,u)$ and finally the part after $u$ of the path $R$ (as above), we form an $s$--$t$ path containing $c$ but neither $b$ nor $a$, contradiction.  Thus $\widetilde d^-(u)=2$. 

Next assuming (1) were not true, we prove (2) still referring to the arc~$c$ and the $s$--$t$ path~$R$ met in previous paragraph.  
Note $|\widetilde\delta^+(u)|\ge1$ because of the arc $c$.  Now if 
$\widetilde\delta^+(u)=\{(u,v)\}$, then $\cor(u,v)$ is on the $s$--$t$ path $R$ and entirely before the arc~$c$ (we cannot have $\cor(u,v)=\cor(c)$ because $c$ is a good arc and the assumption $\widetilde\delta^+(u)=\{(u,v)\}$).  Let $w$ be the final node of  $\cor(u,v)$.  If $w$ had out-degree less than $2$, then $w$ would have in-degree as least $2$ (by the definition of $\cor(u,v)$).  Any arc $(w',w)$ in $\widetilde A \setminus \cor(u,v)$ is on some $s$--$t$ path.  Following the latter from $s$ to $w$, then $R$ to $t$, we get an $s$--$t$ path containing $c$ but avoiding both $a$ and $b$: contradiction.

\medskip

(Sufficiency). 
For any polytope, two of its facets $F$ and $G$ are \underline{not} adjacent if and only if there exists some proper face $K$ such that $F \cap G \subseteq K$ and moreover $K \not\subseteq F$ and $K \not\subseteq G$ (indeed, any facet including $K$ is a facet which includes $F \cap G$ and differs from $F$ and $G$ ). 

Assuming (ii) (assuming (i) leads to similar arguments), either (1) or (2) holds: 

(1)~If $\widetilde d^+(u)\ge2$, let $(u,v)$ and $(u,v')$ be arcs in $\widetilde\delta^+(u)$.  For the face $K$ defined by $x(u,v)\ge0$, we have $F_a \cap F_b \subseteq K$ (because in view of $\widetilde d^-(u)=2$, any $s$--$t$ path containing $(u,v)$ contains $a$ or $b$).  Moreover $K \not\subseteq F_a$ (an $s$--$t$ path including $\cor(a)$ and containing $(u,v')$ gives a vertex in $K$ but not in $F_a$), and similarly $K \not\subseteq F_b$.  Thus the facets $F_a$ and $F_b$ are not adjacent.

(2)~If $\widetilde\delta^+(u)=\{(u,v)\}$, let $w$ be the final node of $\cor(u,v)$.  By assumption, $\widetilde d^+(w)\ge2$, so let $(w,z)$, $(w,z')$ be two arcs in $\widetilde A$.  Letting $K$ be the face defined by $x(w,z)\ge0$, we conclude as in previous paragraph  that the facets $F_a$ and $F_b$ are not adjacent.
\end{proof}

\begin{remark}\label{REM_many_networks}
For many networks $D$, the facets of the flow polytope $\FFF(D)$ are two by two adjacent: it suffices that the network has no node of in- or out-degree equal to $2$.
\end{remark}

\section{Consequences for the Multiple Choice Polytope}
\label{SE_MCP_flows}
We saw in Section~\ref{SEC_MCP} that the multiple choice polytope $\Pmc \CCC$ is affinely isomorphic to the flow polytope $\Flo\CCC$ of the network $\Dlo\CCC=(2^\CCC,\prec,\es,\CCC)$; we keep this notation here, with $n:=|\CCC|$.  
By Proposition~\ref{prop_dim}, the dimension of both $\Flo\CCC$ and $\Pmc \CCC$ equals $2^{n-1}\,(n-2)+1$. Proposition 4 of \cite{Chang_Narita_Saito2022} also  implies this result. 

The vertices of the multiple choice polytope $\Pmc \CCC$ are the points $p^L$, where $L$ is a linear ordering of the set $\CCC$ of alternatives.   The linear mapping (as in \eqref{EQ_f}) 
\begin{equation}\label{EQ_f_rep}
\rho:~ \R^E \to \R^A:~p \mapsto r,
\qquad\text{with } r(T\setminus\{i\},T) :=\; q(i,T)
\end{equation}
 maps the vertex $p^L$ of $\Pmc \CCC$ onto the vertex $\chi^P$ of $\Flo\CCC$, where if $L$ is given by
 \begin{equation}\label{EQ_>_L_rep}
i_1 \quad>_L\quad i_2 \quad>_L\quad \dots \quad>_L\quad i_n
\end{equation}
then $P$ is the $\es$--$\CCC$ path 
 \begin{equation}\label{EQ_path_for_L_rep}
(\es,\{i_1\}),\quad (\{i_1\},\{i_1,i_2\}),\quad \dots,\quad (\{i_1,i_2,\dots,i_{n-1}\},\CCC).
\end{equation}
To determine when two vertices of $\Pmc \CCC$ are adjacent, we rather look at their images by $\rho$ in $\Flo\CCC$.

Proposition~\ref{PROP_path_adj} states when two vertices of any flow polytope are adjacent.  Its particularization to $\Flo\CCC$ translates as follows to the MCP:

\begin{proposition}\label{PROP_MCP_vertex_adj}
For any two linear orderings $L_1$ and $L_2$ of $\CCC$, the vertices $p^{L_1}$ and $p^{L_2}$ of $\Pmc \CCC$ are adjacent if and only if 
\begin{quote}
whenever a nontrivial\footnote{Recall that $A$ is a \textsl{nontrivial} subset of $B$ when $\es \neq A \subset B$.} subset $S$ of $\CCC$ is a beginning set of both $L_1$ and $L_2$,\quad
then $L_1$ and $L_2$ coincide  in $S$ or in $\CCC\setminus S$. \end{quote}
\end{proposition}

For $|\CCC| =2,3$, the graph of the flow polytope $\Flo\CCC$ has diameter $1$ (the polytope is a segment, a $5$-dimensional simplex respectively). 

\begin{corollary}
For $|\CCC| \ge4$, the diameter of the graph of the flow polytope $\Flo\CCC$ equals $2$.
\end{corollary}  

\begin{proof}
Again, we work on the flow polytope $\Flo\CCC$.  Given two $\es$--$\CCC$ paths $P$ and $Q$, we show the existence of a $\es$--$\CCC$ path $R$ such that the vertex $\chi^R$ is adjacent to both vertices $\chi^P$ and $\chi^Q$.
If $(\es,\{i_1\})$ and $(\es,\{j_1\})$ are the two first arcs on respectively $P$ and $Q$, we consider two cases. 
If $i_1=j_1$, we let $R$ be any $\es$--$\CCC$ path with last arc $(\CCC\setminus\{i_1\},\CCC)$.
If $i_1\neq j_1$ we let $R$ be any $\es$--$\CCC$ path with two last arcs $(\CCC\setminus\{i_1,j_1\},\CCC\setminus\{i_1\})$ and $(\CCC\setminus\{i_1\},\CCC)$.    Then no node on $R$, distinct of both $\es$ and $\CCC$, is on $P$ or $Q$
(because the only node on $R$ that contains $i_1$ is $\CCC$, and if $i_1\neq j_1$, the only two nodes on $R$ that contain $j_1$ are $\CCC\setminus\{i_1\}$ and $\CCC$). 
 By Proposition~\ref{PROP_MCP_vertex_adj} $\chi^R$ is adjacent to both $\chi^P$ and $\chi^Q$.
\end{proof}

\medskip

We now turn to the adjacency of facets of the MCP, and again reason on the flow polytope $\Flo\CCC$.  
By Proposition~\ref{PROP_FDI}, a facet of the latter polytope is defined by an inequality $x(a)\ge0$ where $a$ is a good arc in the network $\Dlo\CCC$ (as soon as $|\CCC| \ge 3$ all corridors consist of a single arc, hence distinct good arcs define distinct facets).  For the network~$\Dlo\CCC$, the arc $a=(T\setminus\{i\},T)$ is good if and only if $2 \le |T| \le   |\CCC|-1$.
We deduce that an inequality as in \eqref{EQ_BM}, that is for $(i,T)\in E$ (or $i\in T \in 2^\CCC$)
\begin{equation}\label{EQ_BM_rep}
\sum_{S\in 2^\CCC:\; S \supseteq T} (-1)^{|S\setminus T|} \; p(i,S) \;\ge\; 0,
\end{equation}
defines a facet of $\Pmc \CCC$ if and only if $2 \le |T| \le |\CCC|-1$ (\citealp{Suck1995}, unpublished, and \citealp{Fiorini2004}).
We derive from Proposition~\ref{PRO_non_adjacency_of_facets}:

\begin{proposition}\label{PROP_MCP_facet_adj}
Assume
$|\CCC|\ge 4$.
Consider the two facets of $\Pmc \CCC$ defined by inequalities 
as in \eqref{EQ_BM_rep}, for  the two distinct pairs $(i,T)$ and $(i',T')$ in $E$ with $2 \le  |T|,|T'| \le |\CCC|-1$.  The two facets are adjacent if and only if neither of the two following cases occurs:
\begin{enumerate}[\qquad\rm(i)]
\item $T=\CCC\setminus\{i'\}$ and $T'=\CCC\setminus\{i\}$;
\item $T=T'=\{i,i'\}$.
\end{enumerate}

\end{proposition}

For $n:=|\CCC|\ge4$, it readily follows that the adjacency graph on the collection of facets of $\Pmc \CCC$ is the complete graph on $2\,n\,(2^{n-2}-1)$ nodes minus $n\,(n-1)$ two by two disjoint links; thus the graph has diameter~$2$.
For $n\le3$, the graph is complete.

\subsection{Identifiability in the MCM}
\label{subs_identifiability}

It is well known that the MCM is not identifiable 
 (see \citealp{Falmagne1978}; \citealp{Fishburn1998} collects several results and references).
In terms of \eqref{EQ_captured_by_f}, 
it means the existence of at least one predicted point $p$ in $\Pmc\CCC$ for which there exists more than one point $Pr$ in $\Lambda(\LLL\OOO_\CCC)$ such that $f(Pr)=p$; in this situation, we say that the point $p$ is \textsl{non-identifying}, and the points $Pr$ are \textsl{non-identified}\footnote{As in  \cite{Doignon_Heller_Stefanutti2018}, the term ``non-identifiable'' is currently used in both cases, but we prefer to reserve it to qualify the model.}.  
Proposition~1 in \cite{McClellon2015b} states that all points in the relative interior of $\Pmc\CCC$ are non-identifying.
Theorem~2 in \cite{Turansick2022} characterizes as follows the non-identified points in $\Lambda(\LLL\OOO_\CCC)$, in terms of beginning sets of linear orderings (beginning sets were defined in \eqref{Eq_infty}).  

\begin{proposition}[\citealp{Turansick2022}]\label{prop_Turansick}
 In the MCM, the distribution $Pr$ on $\LLL\OOO_\CCC$ is identified if and only if there is \underline{no} pair of linear orderings $L$, $L'$ of $\CCC$ such that
\begin{enumerate}[\quad\rm(1)]
\item $Pr(L)>0$ and $Pr(L')>0$;
\item there exist alternatives $i$, $j$, $k$ with
\begin{enumerate}[\qquad\rm(a)]
\item $i >_L k$,\quad $j >_L k$,\quad  $i >_{L'} k$,\quad and\quad $j >_{L'} k$;
\item $i \neq j$;
\item $L^-(k) \neq L'^-(k)$;
\item $L^-(i) = L'^-(j)$.
\end{enumerate}
\end{enumerate}
\end{proposition}

Here is a geometric interpretation of Condition~(2) from Proposition~\ref{prop_Turansick}.  Recall that $Pr^L$ designates the distribution on $\LLL\OOO_{\CCC}$ that is concentrated on the linear ordering $L$; in other terms, $Pr^L$ is a vertex of the simplex $\Lambda(\LLL\OOO_{\CCC})$. Moreover, the vertices of the polytope $\Pmc\CCC$ are the images by $f$ of the vertices of $\Lambda(\LLL\OOO_{\CCC})$; we set $p^L=f(Pr^L)$.

\begin{proposition}
The three following conditions on two linear orderings $L$ and $L'$ of $\CCC$ are equivalent:
\begin{enumerate}[\quad\rm(A)]
\item  $L$ and $L'$ satisfy Conditions~{\rm(2)} in Proposition~\ref{prop_Turansick};
\item there exists a nontrivial subset $U$ of $\CCC$ such that 
    \begin{enumerate}
    \item[\qquad($\alpha$)] $U$ is a beginning set of both $L$ and $L'$, \quad and
    \item[\qquad($\beta$)] $L$ and $L'$ do not coincide on $U$ nor on $\CCC\setminus U$;
    \end{enumerate}
\item the vertices $p^L$ and $p^{L'}$ of $\Pmc\CCC$ are \underline{not} adjacent.
\end{enumerate}
\end{proposition}

\begin{proof}

\noindent(A) $\Rightarrow$ (B)
Letting $U=L^-(i)$, we prove that $U$ satisfies ($\alpha$) and ($\beta$).  Necessarily $i\in U$, and because $L^-(i) = L'^-(j)$, also $j \in U$.  Moreover, $i$ and $j$ being distinct and also the smallest elements in $U$ for respectively the orderings $L$ and $L'$, the two orderings do not coincide on $U$.  Next, because by (a) we have $k\notin U$, (c) implies that $L$ and $L'$ do not coincide on $\CCC\setminus U$.

\medskip

\noindent(B) $\Leftarrow$ (A)
Among all the nontrivial subsets $U$ of $\CCC$ satisfying ($\alpha$) and ($\beta$), take the minimum one w.r.t.\ set inclusion.  Then $U=L^-(i)=L'^-(j)$ for some $i$, $j$ in $\CCC$; moreover by the minimality requirement, $i\neq j$.  Because $L$ and $L'$ do not coincide on $\CCC\setminus U$, there must be some alternative $k$ in $\CCC\setminus U$ which is ranked differently by $L$ and $L'$.  The alternatives $i$, $j$ and $k$ ``do the job''.

\medskip

\noindent The equivalence of (B) and (C) is the object of Proposition~\ref{PROP_MCP_vertex_adj}.
\end{proof}

Thus Turansick's result (here Proposition~\ref{prop_Turansick})  states in a hidden way that the point $Pr$ in $\LLL\OOO_{\CCC}$ is identified if and only if for any two linear orderings $L$ and $M$ of~$\CCC$
\begin{align*}
& Pr(L)>0 \;\land Pr(M)>0  \quad\implies\\
&\qquad\qquad \text{ the vertices } p^L \text{ and } p^M \text{ of } \Pmc\CCC \text{ are adjacent}.
\end{align*}
In a future project, we intend to search for a more efficient characterization of adjacency. 

\section{Consequences for some other particular Flow Polytopes}
\label{SE_other}
The multiple choice polytope appears in \cite{Davis-Stober_Doignon_Fiorini_Glineur_Regenwetter2018} as an `extended formulation' for the `linear order polytope' (we refer the reader to this paper for the definitions of technical terms used only in the present section).  Three more flow polytopes appear there, also as extended formulations, these times for the `weak order polytope', the `interval order polytope' and the `semiorder polytope'.   
We provide characterization of the adjacencies of vertices and of facets for the three flow polytopes. 

\subsection{An extended formulation for the weak order polytope}
Consider the network $\Dwo\CCC=(2^\CCC,\subset,\es,\CCC)$, where the arcs are pairs $(S,T)$ of subsets of $\CCC$ with $S \subset T$. 
The $\es$--$\CCC$ path $P$ equal to (where $S_0=\es$ and $S_k=\CCC$)
\begin{equation}
(S_0, S_1),\quad (S_1, S_2), \quad  \dots,\quad  (S_{k-1}, S_k)
\end{equation}
derives from exactly one weak order on $\CCC$ (a \textsl{weak order} is a binary relation which is transitive and complete), namely the weak order $W$ whose equivalence classes are
\begin{equation}\label{Eq_char_wo}
S_1\setminus S_0 \quad\succ_W \quad S_2\setminus S_1 \quad\succ_W \quad \cdots\quad \succ_W \quad S_k\setminus S_{k-1}.
\end{equation}
A \textsl{beginning set} of a weak order $V$ on $\CCC$ is any subset $S$ of $\CCC$ such that $i\in S$ and $i \ge_V j$ implies $j\in S$ (this extends the definition given in \eqref{Eq_infty} for linear orders).
The weak order $W$ characterized in \eqref{Eq_char_wo} is the weak order whose beginning sets are
\begin{equation} 
S_0, \quad S_1,\quad S_2,\quad \cdots,\quad S_k.
\end{equation}
We say that the vertex $\chi^P$ of the flow polytope $\Fwo\CCC$ corresponding to the $\es$--$\CCC$ path $P$ also corresponds to the weak order $W$.

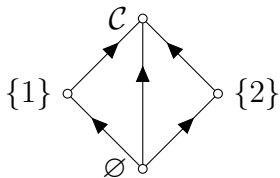
\begin{figure}[ht]
\begin{center}
\begin{tikzpicture}[scale=1]
  \tikzstyle{vertex}=[circle,draw,fill=white, scale=0.3]
  \node (s) at (0,0) [vertex,label=left:$\es$] {};
  \node (t) at (0,2) [vertex,label=left:$\CCC$] {};
  
  \node (u) at (-1,1) [vertex,label=left:{$\{1\}$}] {};
  \node (v) at ( 1,1) [vertex,label=right:{$\{2\}$}] {};

\draw[->-=.7] (s) -- (u); \draw[->-=.7] (s) -- (v); 
\draw[->-=.7] (u) -- (t); \draw[->-=.7] (v) -- (t); 
\draw[->-=.7] (s) -- (t); 
\end{tikzpicture}
\end{center} 
\caption{\label{fig_weak_order}The network in Example~\ref{ex_weak_order}.}
\end{figure}

\begin{example}\label{ex_weak_order}
For $\CCC=\{1,2\}$, the network $\Dwo\CCC=(2^\CCC,\subset,\es,\CCC)$ is displayed in Figure~\ref{fig_weak_order}.  The flow polytope $\Fwo\CCC$ is a triangle. 
\end{example}

Note that for $|\CCC|\ge3$, all corridors of the network $(2^\CCC,\subset,\es,\CCC)$ have size~$1$.

\begin{proposition}\label{prop_WO_vertex_adj}
Assume $|\CCC|\ge3$. The two vertices of the flow polytope $\Fwo\CCC$ corresponding to the two weak orderings $W_1$ and $W_2$ of $\CCC$ are adjacent if and only if when a 
nontrivial subset $S$ of $\CCC$ is a beginning set of both $W_1$ and $W_2$, then $W_1$ and $W_2$ coincide in $S$ or in $\CCC\setminus S$.
\end{proposition}

\begin{corollary}
When $|\CCC|\ge3$, the diameter of the flow polytope $\Fwo\CCC$ equals~$2$.
\end{corollary}

\begin{proof}
The weak order $\CCC\times \CCC$ (with $\CCC$ as its single equivalence class) produces a vertex of $\Fwo\CCC$ which is adjacent to all other vertices.
\end{proof}

\begin{proposition}\label{prop_WO_facet_adj}
Assume $|\CCC|\ge3$. An inequality $x(a)\ge0$, for $a=(S,T)$ with $S \subset T \subseteq \CCC$, defines a facet of the flow polytope $\Fwo\CCC$ if and only if $\es \neq S$ and $T  \neq \CCC$.  Any two facets of $\Fwo\CCC$  are adjacent.
\end{proposition}

More terminology is needed to describe the next two flow polytopes.  To keep the length of this paper (hopefully) acceptable, we state our results without repeating all definitions from \cite{Davis-Stober_Doignon_Fiorini_Glineur_Regenwetter2018}.

\subsection{An extended formulation for the interval order polytope}
For any set $\CCC$ of $n$ alternatives, the network $\Dio \CCC = (N,A,s,t)$  is defined as follows (see Figure~\ref{FIG_int_order} for $|\CCC|=2$):
\begin{eqnarray*}
N & := & \{(X,Y) \st Y \subseteq X \subseteq \CCC\},\\[1mm] 
A & := & \left\{
((X,Y),(Z,T))\in N \times N 
\;\vrule height23pt width1pt depth14pt \, 
\begin{array}{l}
X \subseteq Z,\;
Y \subseteq T,\;\text{and}\\  
\begin{array}{rl}
\text{either~}&|Z| = |X| + 1,\; |T| = |Y|\\ 
\text{or~}&|Z| = |X|,\; |T| = |Y| + 1
\end{array}
\end{array}
\kern-2mm
\right\},\\
s & := & (\es,\es),\\
t & := & (\CCC,\CCC).
\end{eqnarray*}
The flow polytope $\Fio \CCC$ is an extended formulation of the interval order polytope (the vertices of the last polytope are the characteristic vectors of the interval orders on $\CCC$), see
 \cite{Davis-Stober_Doignon_Fiorini_Glineur_Regenwetter2018}.
The numbers of nodes and arcs in the network $\Dio \CCC$ are respectively,  for $n:=|\CCC|$,
\begin{equation}
|N| = 3^n \quad \textrm{and} \quad |A| = 2 \, n \, 3^{n-1}
\end{equation}
(several $(\es,\es)$--$(\CCC,\CCC)$ paths encode the same interval order).

\begin{figure}[ht]
\begin{center}
\begin{tikzpicture}[scale=1]
\tikzstyle{vertex}=[circle,draw,fill=white,scale=0.3]

\scriptsize
  \node (t) at (0,4) [vertex,label=above:{$(\{1,2\},\{1,2\})$}] {};
  
  \node (n12v1) at (-1,3) [vertex,label=left:{$(\{1,2\},\{1\})$}] {};
  \node (n12v2) at ( 1,3) [vertex,label=right:{$(\{1,2\},\{2\})$}] {};
  
  \node (n1v1) at (-2,2) [vertex,label=left:{$(\{1\},\{1\})$}] {};
  \node (n12ve) at ( 0,2) [vertex] {};

  \node (n2v2) at ( 2,2) [vertex,label=right:{$(\{2\},\{2\})$}] {};
  
  \node (n1ve) at (-1,1) [vertex,label=left:{$(\{1\},\es)$}] {};
  \node (n2ve) at ( 1,1) [vertex,label=right:{$(\{2\},\es)$}] {};

  \node (s) at (0,0) [vertex,label=below:{$(\es,\es)$}] {};
  
\draw[->-=.7] (n12v1) -- (t);

\draw[->-=.7] (n12v2) -- (t); 

\draw[->-=.7] (n1v1) -- (n12v1); 
\draw[->-=.7] (n12ve) -- (n12v1); 
\draw[->-=.7] (n12ve) -- (n12v2); 
\draw[->-=.7] (n2v2) -- (n12v2); 

\draw[->-=.7] (n2ve) -- (n12ve);  \draw[->-=.7] (n2ve) -- (n2v2); 

\draw[->-=.7] (n1ve) -- (n1v1);   \draw[->-=.7] (n1ve) -- (n12ve);

\draw[->-=.7] (s) -- (n1ve);      \draw[->-=.7] (s) -- (n2ve); 
\end{tikzpicture}
\end{center} 
\caption{\label{FIG_int_order}The network $\Dio \CCC$ used in the investigation of interval orders, for $|\CCC|=2$.  The label of the central node is $(\{1,2\},\es)$.}
\end{figure}
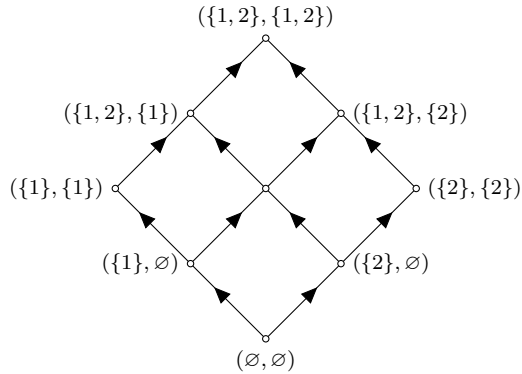

When $|\CCC|\ge3$, all corridors of the network $\Dio \CCC$ have size~$1$.  
For the adjacency of vertices, we cannot tell more than the characterization in Proposition~\ref{PROP_path_adj} (note that the vertices of $\Fio \CCC$ do not have a simple interpretation
while the vertices of $\Flo \CCC$ and $\Fwo \CCC$ exactly correspond to linear orders and weak orders on $\CCC$ respectively; see \citealp{Davis-Stober_Doignon_Fiorini_Glineur_Regenwetter2018}, for more details on $\Fio \CCC$).
For the facets we have:  

\begin{proposition}\label{prop_IO_facet_adj}
Let $a$ be any arc in $\Dio \CCC$, with $|\CCC|\ge3$.  The inequality $x(a)\ge0$ defines a facet $F_a$
of the flow polytope~$\Fio\CCC$ if and only if the arc~$a$ is good, equivalently $a$ is not of any of the four forms, for some $i \in\CCC$,
\begin{gather*}
(\, (\es,\es),\, (\{i\},\es) \,), \qquad (\, (\{i\},\es),\, (\{i\},\{i\}) \,),\\
(\, (\CCC\setminus\{i\},\, \CCC\setminus\{i\} \,), 
(\, \CCC,\, \CCC\setminus\{i\} \,)),\qquad 
(\, (\CCC,\CCC\setminus\{i\}),\, (\CCC,\CCC) \,).
\end{gather*}   
If the two arcs $a$ and $b$ of $\Dio \CCC$ are good, then the two facets $F_a$ and $F_b$ are \underline{not} adjacent if and only if $\{a,b\}$ is, for some distinct alternatives $i$ and $j$, one of the six pairs of arcs shown in Figure~\ref{fig_6_pairs}.
\end{proposition}

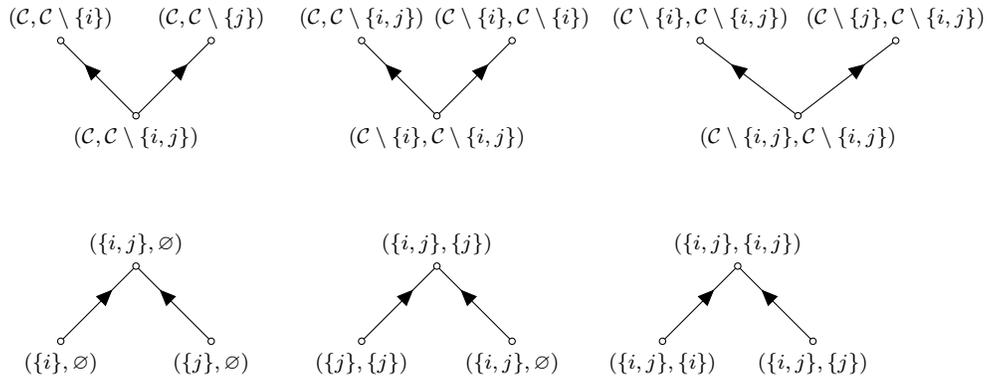
\begin{figure}[ht]
\begin{center}
\begin{tikzpicture}[scale=1]
  \tikzstyle{vertex}=[circle,draw,fill=white,scale=0.3]
  \scriptsize
\node[vertex,label=above:{$(\CCC,\CCC\setminus\{i\})$}] (a)at(-1,3) {};
\node[vertex,label=above:{$(\CCC,\CCC\setminus\{j\})$}] (b) at (1,3) {};
\node (g) at (0,2)[vertex,label=below:{$(\CCC,\CCC\setminus\{i,j\})$}] {};
\draw[->-=.7] (g) -> (a);
\draw[->-=.7] (g) -> (b);

\node (c) at (3,3)[vertex,label=above:{$(\CCC,\CCC\setminus\{i,j\})$}] {};
\node (d) at (5,3)[vertex,label=above:{$(\CCC\setminus\{i\},\CCC\setminus\{i\})$}] {};
\node (h) at (4,2)[vertex,label=below:{$(\CCC\setminus\{i\},\CCC\setminus\{i,j\})$}] {};
\draw[->-=.7] (h) -> (c);
\draw[->-=.7] (h) -> (d);

\node (e) at (7.5,3)[vertex,label=above:{$(\CCC\setminus\{i\},\CCC\setminus\{i,j\})$}] {};
\node (f) at (10.1,3)[vertex,label=above:{$(\CCC\setminus\{j\},\CCC\setminus\{i,j\})$}] {};
\node (i) at (8.8,2)[vertex,label=below:{$(\CCC\setminus\{i,j\},\CCC\setminus\{i,j\})$}] {};
\draw[->-=.7] (i) -> (e);
\draw[->-=.7] (i) -> (f);


\node (j) at (0,0)[vertex,label=above:{$(\{i,j\},\es)$}] {};
\node (m) at (-1,-1)[vertex,label=below:{$(\{i\},\es)$}] {};
\node (n) at (1,-1)[vertex,label=below:{$(\{j\},\es)$}] {};
\draw[->-=.7] (m) -> (j);
\draw[->-=.7] (n) -> (j);

\node (k) at (4,0)[vertex,label=above:{$(\{i,j\},\{j\})$}] {};
\node (o) at (3,-1)[vertex,label=below:{$(\{j\},\{j\})$}] {};
\node (p) at (5,-1)[vertex,label=below:{$(\{i,j\},\es)$}] {};
\draw[->-=.7] (o) -> (k);
\draw[->-=.7] (p) -> (k);

\node (l) at (8,0)[vertex,label=above:{$(\{i,j\},\{i,j\})$}] {};
\node (q) at(7,-1)[vertex,label=below:{$(\{i,j\},\{i\})$}] {};
\node (r) at (9,-1)[vertex,label=below:{$(\{i,j\},\{j\})$}] {};
\draw[->-=.7] (q) -> (l);
\draw[->-=.7] (r) -> (l);
\end{tikzpicture}
\end{center} 
\caption{\label{fig_6_pairs}The six types of pairs of arcs producing pairs of nonadjacent facets of $\Fio \CCC$.}
\end{figure}

\begin{proof}
By Proposition~\ref{PROP_FDI} and because the network $\Dio \CCC$ has more than one $\es$--$\CCC$ path, $x(v)\ge 0$ defines a facet if and only if the arc $v$ is good.   When $|\CCC|\ge 3$, any corridor is formed by a single arc.
Note that a node $(X,Y)$ has in-degree $|X|$ and out-degree $|\CCC\setminus Y|$.  Hence the in-degree of any node $v$ in $\Dio \CCC$ is at least $2$ except when $v$ equals $(\es,\es)$, $(\{i\},\es)$,  or $(\{i\},\{i\})$ for some alternative $i$ 
(here again we need $|\CCC|\ge3$, as testified by Figure~\ref{FIG_int_order}).
Similarly, the out-degree of any node $w$ in $\Dio \CCC$ is at least $2$ except when $w$ equals $(\CCC\setminus\{j\},\CCC\setminus\{j\})$, $(\CCC,\CCC\setminus\{j\})$ or $(\CCC,\CCC)$ for some alternative $j$.
It follows that the only bad arcs are those mentioned in the statement.

Now suppose that the two arcs $a$ and $b$ are good. 
Referring to Proposition~\ref{PRO_non_adjacency_of_facets}, we see that the facets $F_a$ and $F_b$ are \underline{not} adjacent exactly if either $a$ and $b$ have the same initial node, say $u$, with $d^+(u)=2$, or $a$ and $b$ have the same terminal node, say $v$, with $d^-(v)=2$ (here the cases (2) in Proposition~\ref{PRO_non_adjacency_of_facets} cannot occur in view of $|\CCC|\ge3$).
When $|\CCC|\ge3$, the latter happens exactly for any of the six types of arcs displayed in Figure~\ref{fig_6_pairs}.
\end{proof}

\subsection{An extended formulation for the semiorder polytope}
\cite{Davis-Stober_Doignon_Fiorini_Glineur_Regenwetter2018}  introduce still another network $\Dso \CCC = (N,A,\linebreak s,  t)$
with $n:=|\CCC|$,
whose flow polytope makes an extended formulation of the `semiorder polytope'.  The definition of $\Dso \CCC$ goes as follows, where $L + i$ means that we append alternative $i$ at the end of the linear ordering $L$ of some subset of $\CCC$ excluding $i$.  Moreover $L -j$ denotes the removal of $j$ from the ground set of the linear order $L$.  As a convention, the only linear ordering of the empty set is $L=\es$.

\begin{eqnarray*}
N &=& \{(X,Y,L) \st \CCC \supseteq X \supseteq Y,\; L\textrm{ linear ordering of } X\setminus Y\};\\[2mm]
A &=& \{\big((X,Y,L),\, (Z,T,M) \big) \in N^2  \st \\
&&\text{either for some } i\in\CCC\setminus X:\quad
\left\{\begin{array}{lcl}
Z &=& X \cup \{i\},\\ 
T &=& Y,\\
M &=& L + i,
\end{array}\right.\\
&&  \text{or for the alternative $j$ in $X \setminus Y$ which is the first one in $L$}:\\
&&
\phantom{\text{either for some } i\in\CCC\setminus X:\quad}
\left\{\begin{array}{lcl}
Z &=& X,\\ 
T &=& Y \cup \{j\},\\
M &=& L - j;
\end{array}\right.\\
s &=& (\es,\es,\es);\\
t &=& (\CCC,\CCC,\es).
\end{eqnarray*}

Each $(\es,\es,\es)$--$(\CCC,\CCC,\es)$ path is a sequence of $2\,n$ arcs (here, again, $n:=|\CCC|$).
See Figure~\ref{FIG_semiorder}  for $\Dso \CCC$ when $n=2$.

\begin{figure}[ht]
\begin{center}
\begin{tikzpicture}[scale=1]
\tikzstyle{vertex}=[circle,draw,fill=white,scale=0.3]

\scriptsize

  \node (t) at (0,4) [vertex,label=above:{$(\{1,2\},\{1,2\},\es)$}] {};
  
  \node (n12v1) at (-1,3) [vertex,label=left:{$(\{1,2\},\{1\},L)$}] {};
  \node (n12v2) at ( 1,3) [vertex,label=right:{$(\{1,2\},\{2\},L)$}] {};
  
  \node (n1v1) at (-2,2) [vertex,label=left:{$(\{1\},\{1\},\es)$}] {};
  
  \node (n12veA) at (-0.5,2) [vertex] {};
  \node (n12veB) at ( 0.5,2) [vertex] {};

  \node (n2v2) at ( 2,2) [vertex,label=right:{$(\{2\},\{2\},\es)$}] {};
  
  \node (n1ve) at (-1,1) [vertex,label=left:{$(\{1\},\es,L)$}] {};
  \node (n2ve) at ( 1,1) [vertex,label=right:{$(\{2\},\es,L)$}] {};

  \node (s) at (0,0) [vertex,label=below:{$(\es,\es,\es)$}] {};
  
\draw[->-=.7] (n12v1) -- (t);
\draw[->-=.7] (n12v2) -- (t); 

\draw[->-=.7] (n1v1) -- (n12v1); 
\draw[->-=.7] (n12veA) -- (n12v1); 
 
\draw[->-=.7] (n12veB) -- (n12v2); 
\draw[->-=.7] (n2v2) -- (n12v2); 

\draw[->-=.5] (n2ve) -- (n12veB);
\draw[->-=.7] (n2ve) -- (n2v2); 

\draw[->-=.7] (n1ve) -- (n1v1);
\draw[->-=.5] (n1ve) -- (n12veA);

\draw[->-=.7] (s) -- (n1ve); \draw[->-=.7] (s) -- (n2ve); 
\end{tikzpicture}
\end{center} 
\caption{\label{FIG_semiorder}The network $\Dso \CCC$ used in the investigation of semiorders, for $|\CCC|=2$.  When $|X \setminus Y| \le 1$, the linear ordering of $X \setminus Y$ is obvious; we simply write $\es$ or $L$ for it.  The labels of the central nodes are 
 $(\{1,2\},\es,1 <_L 2)$ and $(\{1,2\},\es, 2 <_L 1)$ respectively.}
\end{figure}
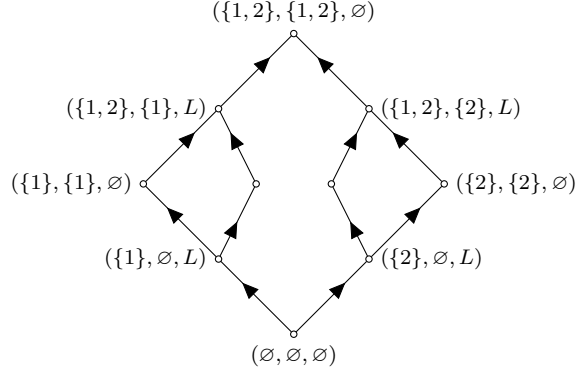

\begin{lemma}\label{LEM_SO_degrees}
For any node $(X,Y,L)$ in the network $\Dso \CCC$,
\begin{align}
\widetilde d^+(X,Y,L) \;&=\:
\begin{cases}
n-|X| &\text{ if } X=Y,\\
n-|X|+1 \quad&\text{ if } X \supset Y;
\end{cases}  \\
\widetilde d^-(X,Y,L) \;&=\: 
\begin{cases}
|Y| &\text{ if } X = Y,\\
|Y|+1 \quad&\text{ if } X \supset Y.
\end{cases}
\end{align}
\end{lemma}

\begin{proof}
The first two equations derive from the definition of arcs with tail $(X,Y,L)$.  To derive the last two equations, rewrite the definition as follows.  For two nodes $(Z,T,M)$ and $(X,Y,L)$, the pair $\big((Z,T,M),\, (X,Y,L) \big)$ is an arc if and only if
\begin{equation} 
\text{for 
$i\in\ X\setminus Y$ which is the last for $L$}:\;
\left\{\begin{array}{lcl}
Z &=& X \setminus \{i\},\\ 
T &=& Y,\\
M &=& L - i,
\end{array}\right.
\end{equation}
or
\begin{equation}
\text{for 
some $j$ in $Y$}: \\
\left\{\begin{array}{lcl}
Z &=& X,\\ 
T &=& Y \setminus \{j\},\\
M &=& j + L.
\end{array}\right.
\end{equation}
\end{proof}

Here again,  as for the interval order case, there is no more about adjacency of vertices that we can say beside Proposition~\ref{PROP_path_adj}.  
We thus turn to the adjacency of facets.

\begin{proposition}\label{pro_SO_good}
Assume $|\CCC|\ge3$.  
All corridors of $\Dso \CCC$ consist of either one arc or two arcs. The corridors of size~$2$ have central nodes of the form $(\CCC,\es,L)$, for some linear ordering $L$ of $\CCC$; both of their arcs are good.
An arc of $\Dso \CCC$ is good if and only if it is not of any of the following types:
\begin{align}
\label{eq_alpha}\tag{$\alpha$}
((X,\es,L),(X\cup\{i\},\es,L+i)) &\qquad\text{where } X\subset\CCC,\, i\in \CCC\setminus X;
\\
\label{eq_beta}\tag{$\beta$}
((\CCC\setminus\{i\},\CCC\setminus\{i\},\es),(\CCC,\CCC\setminus\{i\},L)) &\qquad \text{where } i\in\CCC;\\
\label{eq_gamma}\tag{$\gamma$} 
(X,\es,L),(X,\{j\},L-j) &\qquad \text{where } 
j\in X \subseteq \CCC\\
\label{eq_delta}\tag{$\delta$}
((\CCC,Y,L),(\CCC,Y\cup\{j\},L-j) &\qquad\text{where } Y \subset \CCC ,\, j \in \CCC\setminus Y.
\end{align}
\end{proposition}


\begin{proof}
By Lemma~\ref{LEM_SO_degrees}, the only nodes of $\Dso \CCC$ having both in- and out-degree $1$ are the $(\CCC,\es,L)$'s with $L$ any linear ordering of $\CCC$.  So the corridors are of size $1$ or $2$, and the corridors of size $2$ have $(\CCC,\es,L)$ as their middle nodes.  Note moreover that each arc in a corridor of size $2$ is good because both the terminal node $(\CCC,\{j\},L-j)$ 
of the corridor (with $j$ the first element in $L$) has in-degree at least~$2$ and the initial node $(\CCC\setminus\{i\},\es,L-i)$
of the corridor (with $i$ the last element in $L$) has out-degree at least~$2$.

According to the definition of $\Dso \CCC$, there are two types of arcs,
 which we now review for badness:

\medskip

\noindent$\triangleright$~ If the arc $((X,Y,L),(X\cup\{i\},Y,L+i))$ is bad (where $i\in\CCC\setminus X$), then
$d^-((X\cup\{i\},Y,L+i)) = 1$ or $d^+((X,Y,L)) = 1$.
By Lemma~\ref{LEM_SO_degrees}, in the first case, ($X\cup\{i\}=Y$ and $|Y|=1$) or ($Y=\es$ and $X\cup\{i\} \supset Y$).  The first eventuality being impossible because by assumption $X \supseteq Y$, we get \eqref{eq_alpha}.  In the second case, again by Lemma~\ref{LEM_SO_degrees} and with $n:=|\CCC|$, we have ($X=Y$ and $|X|=n-1$) or ($X=\CCC \supset Y$).  The second eventuality being impossible (because we need $i$ in $\CCC\setminus X$), we get \eqref{eq_beta}.

\medskip

\noindent$\triangleright$~ If the arc $(X,Y,L),(X,Y\cup\{j\},L-j)$ is bad (where $j \in X \setminus Y$ is the first element in the linear ordering $L$ of $X \setminus Y$), then
$d^-((X,Y\cup\{j\},L-j)) = 1$ or $d^+((X,Y,L)) = 1$.
By Lemma~\ref{LEM_SO_degrees}, in the first case, ($X=Y\cup\{j\}$ and $|Y\cup\{j\}|=1$) or ($X \supset Y \cup \{j\}$ and $Y = \es$), so we get \eqref{eq_gamma}.
In the second case, ($X=Y$ and $|X|=n-1$) or ($X \supset Y$ and $|X|=n$).  The first eventuality being impossible (in view of $j\in X \setminus Y$), we get \eqref{eq_delta}.
\end{proof}

\begin{proposition} 
Assume $n:=|\CCC|\ge3$.
Take the two facets of $\Fso \CCC$ defined by the inequalities $x(a)\ge0$ and $x(b)\ge0$, where $a$ and $b$ are two good arcs.
The two facets are \underline{not} adjacent if and only if the corridors  $\cor(a)$ and $\cor(b)$ 
\begin{enumerate}[\quad\rm(i)]
\item have the same tail of the form either $(X,X,\es)$ with $|X|=n-2
\ge 2$,
or $(X,Y,L)$ with $|X|=n-1$ and $X\neq Y$, 
\item or they have the same head of the form either $(X,X,\es)$ with $|X|=2$ and $n\ge4$, 
or $(X,Y,L)$ with $|Y|=1$ and $X\neq Y$.
\end{enumerate}
\end{proposition}

\begin{proof}
Refer to Proposition~\ref{PRO_non_adjacency_of_facets} and Proposition~\ref{pro_SO_good}.
\end{proof}




\end{document}